\documentclass[a4paper]{article}
\usepackage[utf8]{inputenc}
\usepackage{newunicodechar}
\newunicodechar{∶}{:} 
\usepackage{amsfonts}
\usepackage{latexsym,amssymb}
\usepackage[fleqn]{amsmath}
\usepackage{amsbsy}
\usepackage{amsopn, amstext}
\usepackage{graphicx, color, epstopdf}
\usepackage{threeparttable}
\usepackage{authblk}
\usepackage{amsthm}
\usepackage{float}
\usepackage{subfig}
\usepackage{multirow}

\newtheorem{Lemma}{Lemma}
\newtheorem{theorem}{Theorem}
\newtheorem{remark}{Remark}
\newtheorem{example}{Example}

\numberwithin{equation}{section}
\numberwithin{Lem}{section}
\numberwithin{Defi}{section}
\numberwithin{Theo}{section}
\numberwithin{Rem}{section}
\numberwithin{Coro}{section}
\numberwithin{Fig}{section}

\title{A Müntz-collocation spectral method for weakly singular Volterra delay-integro-differential equations}
\author[1]{Borui Zhao\thanks{ m202270040@hust.edu.cn}}
\affil[1]{School of Mathematics and Statistics, Huazhong University of Science and Technology, Wuhan 430074, China}

\date{}

\begin{document}
\maketitle
\begin{abstract}
	A Müntz spectral collocation method is implemented for solving weakly singular Volterra integro-differential equations (VDIEs) with proportional delays. After constructing the numerical scheme to seek an approximate solution, we derive error estimates in a weighted $L^2$ and $L^{\infty}$-norms. A rigorous proof reveals that the proposed method can handle the weak singularity of the exact solution at the initial point $t=0$, with the numerical errors decaying exponentially in certain cases. Moreover, several examples will illustrate our convergence analysis.
\end{abstract}

\section{Introduction}\label{section_Introduction}
Volterra integro-differential equations (VIDEs) are widely used in mathematical models such as population dynamics\cite{application_population} and viscoelastic phenomena\cite{application_粘弹性现象}. In this paper, we consider weakly singular VIDEs with proportional delays
	\begin{align} 
		y^{\prime}(t)=&a_1(t) y(t)+b_1(t) y(\varepsilon t)+f_1(t)+\left(\mathcal{K}_1y\right)(t)+\left(\mathcal{K}_2y\right)(t),  \   \
		t \in I_0:=[0,T] ,\label{eq_prime}\\
		y(0)=& y_0.\label{eq_origin}
	\end{align}
where$\quad 0 \leq \mu<1$, $a_1{(t)}$, $b_1{(t)}$, $f_1{(t)}$, $K_1(t, s)$, $K_2(t, \tau)$ are given smooth functions. $\left(\mathcal{K}_1y\right)(t)= \int_0^t(t-s)^{-\mu}K_1(t, s)$$y(s) d s$, $\left(\mathcal{K}_2y\right)(t)=\int_0^{\varepsilon t}(\varepsilon t-\tau)^{-\mu}K_2(t, \tau) y(\tau) d \tau.$

Spectral methods plays a crucial role for smooth problems and can provide exponential convergence and excellent error estimates. C. Huang, T. Tang, and Z. Zhang\cite{spectral_2} gave a supergeometric convergence for Volterra and Fredholm integral equations through spectral collocation method.  A spectral collocation method for weakly singular VIEs have been innovatively constructed by Chen and Tang\cite{spectal_TangTao}. Based on this, Chen and others have proposed a spectral Jacobi  collocation method for VIDEs and VIDEs with proportional delays \cite{spectral_Chen_WVIDEs}-\cite{spectral_Chen_WVIDEs_delay} for VIDEs and with proportional delays, and the references therein. These works obtained the exponential convergence results under $L^2$-norm. \cite{spectral_4} has revealed spectral Petrov-Galerkin methods for solving VIDEs. Besides, there are many other strategies for solving VIDEs with weakly singular kernels, such as variable transformation method\cite{spectral_VT_1}-\cite{spectral_VT_3}, the solution to the equation will have better regularity. 

For delay VIEs and VIDEs, Sheng has established a series of works\cite{delay_Sheng_1}-\cite{delay_Sheng_3}, where he provide convergence analysis of the hp-version of the multistep spectral collocation method. Sinc, Tau and Petrov–Galerkin methods are introduced in the the papers\cite{VIDEs_delays_1}-\cite{VIDEs_delays_5} to solve VIDEs. Recently, \cite{WVIDEs_hp_MaZheng} and \cite{WVIDEs_hp_QinYu} have obtained an hp-version error bound to solve weakly singular VIDEs and weakly singular VIDEs with vanishing delays, achieving the exponential rate of convergence.

Hou and others\cite{spectral_3}\cite{Muntz} have applied fractional Jacobi polynomials to VIEs and obtained excellent numerical results. There
have been also some recent developments such as \cite{3VIEs_MaZheng} and \cite{CVIDEs_MaZheng}. The purpose of this paper is to design a fractional polynomial collocation method with a fractional coefficient $\lambda(0 < \lambda \leq 1)$ for solving the second kind of VIDEs with proportional delays, and to provide a convergence analysis. The solutions of the equations \eqref{eq_prime} often exhibit weak singularity at the initial point $t=0$. We can construct numerical solutions in the form of $\sum_{m=0}^{N}a_mt^{m\lambda}$\cite{Ma毕业论文}, and choose an appropriate $\lambda$ in order to counteract this singularity.

Hou and others\cite{spectral_3}\cite{Muntz} have applied fractional Jacobi polynomials to VIEs, achieving good numerical results. Recent developments include \cite{3VIEs_MaZheng} and \cite{CVIDEs_MaZheng}. This paper aims to design a fractional polynomial collocation method with a fractional coefficient $\lambda$ (where $0 < \lambda \leq 1$) for solving VIDEs of the second kind with proportional delays, and to offer a convergence analysis. A weak singularity at the initial point $t=0$ is exhibited for the solutions to the equations \eqref{eq_prime}. We can construct numerical solutions in the form of $\sum_{m=0}^{N}a_mt^{m\lambda}$\cite{Ma毕业论文}.

This paper is organized as follows. In Section \ref{section_Preliminaries}, some function spaces and useful lemmas will be intorduced in order to demenstrate convergence results. In Section \ref{section_Numerical Scheme}, a fractional $Jacobi$ collocation method will be constructed to estimate the exact solutions of  the equation \eqref{eq_prime}. In Section \ref{section_Convergence Analysis}, under the weighted $L^2_{\omega^{\alpha,\beta}}$ and $L^{\infty}$-norm, we will give a rigorous proof for the convergence analysis. In Section \ref{section_Numerical experiments}, several numerical examples will prove the theoretical analysis of Section 5.

\section{Preliminaries}\label{section_Preliminaries}
From now on, $C$ represents a generic positive constant that is independent of $N$, the number of selected $Jacobi-Gauss$ points. $N$ is a suffiecent large positive integer. Let $I:=[0,1]$.

\subsection{Function spaces}
Let $r\geq0$ and $\kappa \in[0,1]$. Define a space denoted by $C^{r, \kappa}(I)$, in which all functions whose $r$-th derivatives are Hölder continuous with exponent $\kappa$. Its norm is defined as:
	\begin{align}
		\|f(\theta)\|_{r, \kappa}=\max _{0 \leq i \leq r} \max _{\theta \in I}\left|f^{(i)}(\theta)\right|+\max _{0 \leq i \leq r} \sup _{\theta, \eta \in I, \theta \neq \eta} \frac{\left|f^{(i)}(\theta)-f^{(i)}(\eta)\right|}{|\theta-\eta|^\kappa} ,f(\theta) \in C^{r, \kappa}(I) .\label{eq_C_rk}
	\end{align}
When $\kappa = 0$, $C^{r,0}(I)$ is called the space of functions with $r$ continous derivatives on I,which equals to the common space $C^{r}(I)$.\\
Define the $\lambda$-polynomial space \cite{Muntz}:
	$$
	P_n^\lambda\left(\mathbb{R}^{+}\right):=\operatorname{span}\left\{1, \theta^\lambda, \theta^{2 \lambda}, \ldots, \theta^{n \lambda}\right\},
	$$
where $\mathbb{R}^{+}=[0,+\infty), 0<\lambda \leq 1$. $n$-th $\lambda$-polynomials can be presented as 
	$$
	p_n^\lambda(\theta):=k_n \theta^{n \lambda}+k_{n-1} \theta^{(n-1) \lambda}+\cdots+k_1 \theta^\lambda+k_0, \quad k_n \neq 0, \theta \in \mathbb{R}^{+}.
	$$
The sequence of $\{p_n^\lambda\}_{n=0}^{\infty}$ is called to be orthogonal in $L_{\omega}^{2}(I)$ if
	$$
		\left(p_m^\lambda, p_n^\lambda\right)_\omega=\int_0^1 p_n^\lambda(\theta) p_m^\lambda(\theta) \omega(\theta) d \theta=\gamma_m \delta_{m,n},
	$$
where $\gamma_n=\left\|p_n^\lambda\right\|_{0, \omega}^2:=\left(p_n^\lambda, p_n^\lambda\right)_\omega \text { and } \delta_{m, n} \text { is the } Kronecker\text{ delta}.$\\
Then we define $P_n^\lambda(I)$-space which satisfies
	$$
		P_n^\lambda(I):=\operatorname{span}\left\{p_0^\lambda, a_1^\lambda, \ldots, p_n^\lambda\right\},
	$$
This is a special type of Müntz space\cite{Ma毕业论文}, specifically, $P_n^1(I)$ represents the space formed by polynomials of degree no more than $n$.\\
Now we introduce the non-uniformly Jacobi-weighted Sobolev space\cite{Muntz}
	$$
	B_{\alpha, \beta}^{m,1}(I)=\left\{u(\theta): \partial_\theta^k u(\theta) \in L_{\omega^{\alpha+k, \beta+k}}^2(I), 0 \leq k \leq m\right\},m \in \mathbb{N}
	$$
endowed with the inner product, semi-norm and norm:
		$$
	\begin{aligned}
		& (u, v)_{B_{\alpha, \beta}^{m,1}}=\sum_{k=0}^m\left(\partial_\theta^k u, \partial_\theta^k v\right)_{\omega^{\alpha+k, \beta+k,1}},\\
		& |u|_{B_{\alpha, \beta}^{m,1}}=\left\|\partial_\theta^m u\right\|_{0,\omega^{\alpha+m, \beta+m,1}}, \quad\|u\|_{B_{\alpha, \beta}^{m,1}}=(u, u)_{B_{\alpha, \beta}^{m,1}}^{\frac{1}{2}} .
	\end{aligned}
	$$

\subsection{Fractional Jaboci polynomials}
The fractional $Jaboci$ polynomials are derived from the standard $Jacobi$ polynomials via a variable transformation, which effectively maps the interval $[-1,1]$ to $[0,1]$. The relationship between the two is as follows\cite{Muntz}
	$$
	J_n^{\alpha, \beta, \lambda}(\theta)=J_n^{\alpha, \beta}\left(2 \theta^\lambda-1\right), \quad \forall \theta \in I,\alpha, \beta>-1,0<\lambda \leq 1.
	$$
The form of the classical $Jacobi$ polynomials $J_n^{\alpha, \beta}(\theta)$ is 
	$$
	J_n^{\alpha, \beta}(\theta)=\frac{\Gamma(n+\alpha+1)}{n!\Gamma(n+\alpha+\beta+1)} \sum_{k=0}^n\binom{n}{k} \frac{\Gamma(n+k+\alpha+\beta+1)}{\Gamma(k+\alpha+1)}\left(\frac{\theta-1}{2}\right)^k .
	$$
So fractional $Jaboci$ polynomials can be readily derived,
	$$
	J_n^{\alpha, \beta, \lambda}(\theta)=\frac{\Gamma(n+\alpha+1)}{n!\Gamma(n+\alpha+\beta+1)} \sum_{k=0}^n\binom{n}{k} \frac{\Gamma(n+k+\alpha+\beta+1)}{\Gamma(k+\alpha+1)}\left(\theta^\lambda-1\right)^k .
	$$
The relationship between the two sets of $Jacobi$-$Gauss$ quadrature nodes and their corresponding weights also exist. Let the standard $Jacobi$-$Gauss$  quadrature nodes be $\lbrace t_i \rbrace_{j=0}^{N}$, with weights $\lbrace w_j \rbrace_{j=0}^{N}$. $\lbrace \theta_{j},\omega_{j} \rbrace_{j=0}^{N}$ are denoted to be the fractional $Jacobi$-$Gauss$  quadrature nodes and weights on $I_0 := [0,1]$. We demonstrate the relationship:
$$
\theta_{j}=(\frac{t_j+1}{2})^{\frac{1}{\lambda}},\omega_{j}=2^{-(\alpha+\beta+1)}w_j,\quad j=0,...,N.
$$
To simplify the notation, we denote the fractional $Jabobi-Gauss$ quadrature nodes of degree $n$ as $\lbrace \theta_j \rbrace_{j=0}^{N}$, with corresponding weights denoted as $\lbrace \omega_j \rbrace_{j=0}^{N}$ in the following text. The weight function is as follows:
	\begin{align}\label{eq_wight_equation}
		\qquad\qquad\qquad\qquad\qquad
		\omega^{\alpha, \beta, \lambda}(\theta):=\lambda\left(1-\theta^\lambda\right)^\alpha \theta^{(\beta+1) \lambda-1}.
	\end{align}
Note that $\left\{F_{j,\lambda}\left(\theta\right)\right\}_{j=0}^N$ represents the generalized Lagrange interpolation basis functions
	$$
	F_{j,\lambda}\left(\theta\right)=\prod_{i=0, i \neq j}^N \frac{\theta^\lambda-\theta_i^\lambda}{\theta_j^\lambda-\theta_i^\lambda}, \quad 0 \leq j \leq N,
	$$
where $\theta_0<\theta_1<\cdots<\theta_{N-1}<\theta_N$ are zeros of the fractional $Jacobi$ polynomials $J_{N+1}^{\alpha, \beta, \lambda}(\theta)$ and $F_{j,\lambda}\left(\theta\right)$ clearly satisfy
	$$
	F_{j,\lambda}\left(\theta_i\right)=\delta_{i j} .
	$$
We define the generalized interpolation operator $I_{N, \lambda}^{\alpha, \beta}$ as follows
	\begin{align}\label{eq_Interpolation_operator_sum}
	\qquad \qquad  I_{N, \lambda}^{\alpha, \beta} v(\theta)=\sum_{j=0}^N v\left(\theta_j\right) F_{j,\lambda}\left(\theta\right)=\sum_{j=0}^N v\left(z_j^{\frac{1}{\lambda}}\right) F_{j,1}\left(z\right)=I_{N, 1}^{\alpha, \beta}v(z^{\frac{1}{\lambda}}),\quad \theta=z^{\frac{1}{\lambda}}.
\end{align}

\subsection{Elementary lemmas}
In this section, we will present the lemmas required for the convergence analysis in section \ref{section_Convergence Analysis}.

\begin{Lemma}\label{lemma_Gronwall}
	$($Gronwall inequality$)$ Assume that
	$$
	f(\theta) \leq g(\theta)+C \int_0^{\theta} f(\eta) d \eta, \quad 0 \leq \theta \leq 1,
	$$\\
	If $f(\theta), g(\theta)$ are non-negative integrable functions on $[0, 1]$ and $C > 0$, then there exists $L>0$ such that:
	$$
	f(\theta) \leq g(\theta)+L \int_0^{\theta} g(\eta) d \eta, \quad 0 \leq \theta \leq 1 .
	$$
\end{Lemma}

\begin{Lemma}\label{lemma_Chen_Lemma3.5_norm}$($see $\cite{spectral_Chen_WVIDEs_delay}$$)$
If $E(\theta)$ is a nonnegative integrable function which satisfies
	$$
	\begin{aligned}
		E(\theta)\leq J(\theta)+L\int_{0}^{\theta}E(\eta)d\eta,0\leq\theta\leq1,\notag
	\end{aligned}
	$$
where $J(\theta)$ is a integrable function and $L$ is a positive constant. Then the following conclusion holds
	$$
	\begin{aligned}
		&\left\|E(\theta)\right\|_{\infty}\leq C\left\|J(\theta)\right\|_{\infty} ,\notag\\
		&\left\|E(\theta)\right\|_{0,\omega^{\alpha,\beta,1}}\leq C\left\|J(\theta)\right\|_{0,\omega^{\alpha,\beta,1}}.\notag
	\end{aligned}
	$$
\end{Lemma}

\begin{Lemma}$($ see $\cite{Lemma_linear_operator_1}$,$\cite{Lemma_linear_operator_2}$$)$
	If $r$ is a nonnegative integer and $\kappa \in(0,1)$, there exists a linear operator $\mathcal{K}_N$ that maps $C^{r, \kappa}(I)$ to $P_N^1(I)$, such that
	\begin{align}\label{eq_Crk}
			\left\|v-\mathcal{K}_N v\right\|_{\infty} \leq C_{r, \kappa} N^{-(r+\kappa)}\|v\|_{r, \kappa}, \quad v \in C^{r, \kappa}(I),
	\end{align}
	$C_{r, \kappa}$ is a constant that may depend on $r$ and $\kappa$. For the linear weakly singular kernel integral operators $\mathcal{K}_i$ (where $i = 1,2$) defined in the previous section:
	\begin{align}\label{lemma_Operator}
		&(\mathcal{K}_1 v)(t)=\int_0^t(t-s)^{-\mu}K_1(t, s) v(s) d s,\\
		&(\mathcal{K}_2 v)(t)= \int_0^{\varepsilon t}(\varepsilon t-s)^{-\mu}K_2(t, \tau) v(\tau) d \tau.
	\end{align}
	$K_i$ belongs to $C(I \times I)$ and $K_i(t, t) \neq 0$ for all $t \in I$. For any $0<\kappa<1-\mu$, we will prove that $\mathcal{K}_i$ is a linear operator mapping from $C(I)$ to $C^{0, \kappa}(I)$.
\end{Lemma}

\begin{Lemma}\label{lemma_J6org}
When $0<\kappa<1-\mu$, for any function $v \in C(I)$ and $K_i \in C(I \times I)$ with $K_i(\cdot, s) \in C^{0, \kappa}(I),(i=1,2)$, there is
	$$
	\begin{aligned}
		\notag \frac{|(\mathcal{K}_i v)(\theta_1)-(\mathcal{K}_i v)(\theta_2)|}{|\theta_1-\theta_2|^{\kappa}} \leq C \max _{\theta_1 \in I}|v(\theta_1)|, \quad \forall \theta_1 ,\theta_2 \in I, \theta_1 \neq \theta_2 .
	\end{aligned}
	$$
	Thus it can be inferred that
	$$
	\begin{aligned}
		\|\mathcal{K}_i v\|_{0, \kappa} \leq C\|v\|_{\infty}, \quad 0<\kappa<1-\mu .\notag
	\end{aligned}
	$$
\begin{proof}
Without loss of generality, we will only prove the case of $\mathcal{K}_2 v$. We can assume that $0 \leq \theta_2 < \theta_1 \leq 1$. Let $K_2$ denote the maximum value over $\eta \in I$ that satisfies $\|K_2(\cdot, \varepsilon\eta)\|_{0,\kappa}$. We define the Beta function as $B(\alpha, \beta)$. 
		\begin{align}
			&\frac{\left|\left(\mathcal{K}_2 v\right)(\theta_1)-\left(\mathcal{K}_2 v\right)(\theta_2)\right|}{|\theta_1-\theta_2|^{-\kappa}}\notag\\
			&= \frac{\left|\int_0^{\varepsilon\theta_2}  (\varepsilon \theta_2-\eta)^{-\mu}  K_2(\theta_2, \eta) v(\eta) \, d\eta \right.
				-\left.\int_0^{\varepsilon\theta_1}  (\varepsilon \theta_1-\eta)^{-\mu}  K_2(\theta_1, \eta) v(\eta) \, d\eta\right| }{|\theta_1-\theta_2|^{-\kappa}}\notag \\
			&\leq\|v\|_{\infty}|\theta_1-\theta_2|^{-\kappa}\left| \int_0^{\varepsilon\theta_2} (\varepsilon \theta_2-\eta)^{-\mu}  K_2(\theta_2, \eta) \right.-\left. (\varepsilon \theta_1-\eta)^{-\mu} K_2(\theta_1, \eta) \, d\eta  \right| \notag \\ 
			&+ \|v\|_{\infty}|\theta_1-\theta_2|^{-\kappa}\int_{\varepsilon\theta_2}^{\varepsilon\theta_1}  (\varepsilon \theta_1-\eta)^{-\mu} \left|K_2(\theta_1, \eta)\right| \, d\eta \notag \\
			&= G_1+G_2\notag \\
			&\leq  G_1^{(1)}+G_2^{(1)}+G_2.
		\end{align}
where
		\begin{align}
			&G_1^{(1)}=\|v\|_{\infty}\left(\theta_1-\theta_2\right)^{-\kappa}
			\int_0^{\varepsilon\theta_2}
			\left|
		\left(\varepsilon \theta_2-\eta\right)^{-\mu}-
			\left(\varepsilon \theta_1-\eta\right)^{-\mu}
			\right|
			\left|
			K_2\left(\theta_2,\eta\right)
			\right|
			d \eta,\notag\\
			&G_1^{(2)}=\|v\|_{\infty}\left(\theta_1-\theta_2\right)^{-\kappa}
			\int_0^{\varepsilon\theta_2}
			\left(\varepsilon \theta_1-\eta\right)^{-\mu}
			\left|
			K_2\left(\theta_2,\eta\right)-K_2\left(\theta_1,\eta\right)
			\right|
			d\eta,\notag\\
			&G_{2}=\|v\|_{\infty}\left(\theta_1-\theta_2\right)^{-\kappa}\int_{\varepsilon\theta_2}^{\varepsilon\theta_1}\left(\varepsilon \theta_1-\eta\right)^{-\mu}\left|K_2\left(\theta_1, \eta\right)\right| d \eta\notag.
		\end{align}
We further analyze $G_1^{(1)}, G_1^{(2)}, G_{2}$, we have
		\begin{align}
			G_1^{(1)} 
			&\leq K_2\|v\|_{\infty}\left(\theta_1-\theta_2\right)^{-\kappa}\left| \int_{0}^{\varepsilon\theta_2}\left(\varepsilon \theta_2-\eta\right)^{-\mu} \, d\eta \right. -\left. \int_{0}^{\varepsilon\theta_1}\left(\varepsilon \theta_1-\eta\right)^{-\mu} \, d\eta \right.
			+ \left. \int_{\varepsilon\theta_2}^{\varepsilon\theta_1}\left(\varepsilon \theta_1-\eta\right)^{-\mu} \, d\eta \right| \notag\\
			&\leq C\|v\|_{\infty}\left(\theta_1-\theta_2\right)^{-\kappa}\int_{\varepsilon\theta_2}^{\varepsilon\theta_1}{\left(\varepsilon\theta_1-\eta\right)^{-\mu}}\, d\eta \notag\\
			&\leq C\|v\|_{\infty}\left(\theta_1-\theta_2\right)^{-\kappa}\varepsilon^{2-\mu}\left(\theta_1-\theta_2\right)^{1-\mu}
			\left.\int_{0}^{1}\left(1-\xi\right)^{-\mu}\left(\theta_2+\left(\theta_1-\theta_2\right)\xi\right)d \xi\right.\notag\\
			&\leq C\|v\|_{\infty}\left(\theta_1-\theta_2\right)^{1-\mu-\kappa}\left.\int_{0}^{1}\left(1-\xi\right)^{-\mu}\left(\theta_1 \right)d \xi\right.\notag\\
			&\leq C\|v\|_{\infty}B(1-\mu,1)\notag\\
			&\leq C\|v\|_{\infty}.\label{eq_M1(1)}
		\end{align}
When \eqref{eq_C_rk} satisfies $r=0$, it can be inferred similarly
		\begin{align}
			G_1^{(2)}=&\|v\|_{\infty}
			\int_0^{\varepsilon\theta_2}
			\left(\varepsilon \theta_1-\eta\right)^{-\mu}
			\frac{\left|
				K_2\left(\theta_2,\eta\right)-K_2\left(\theta_1,\eta\right)
				\right|}{\left(\theta_1-\theta_2\right)^{\kappa}}
			d\eta\notag\\
			&\leq \underset{\eta\in I}{\max}\|K_2(\cdot,\eta)\|_{0,\kappa}\|v\|_{\infty}
			\int_0^{\varepsilon\theta_1}
			\left(\varepsilon\theta_1-\eta\right)^{-\mu} 
			d\eta\notag\\
			&\leq \underset{\eta\in I}{\max}\|K_2(\cdot,\eta)\|_{0,\kappa}\|v\|_{\infty}B(1-\mu,1)\notag\\
			&\leq C\|v\|_{\infty}.\label{eq_M1(2)} 
		\end{align}
Similar with the proof process for $M_1^{(1)}$, we can obtain that
		\begin{align}
			G_2&\leq C\|v\|_{\infty}B(1-\mu,1)\notag\\
			&\leq C\|v\|_{\infty}.\label{eq_M2}
		\end{align}
It is clear from \eqref{eq_M1(1)}-\eqref{eq_M2} that \eqref{eq_C_rk} holds, and it is obvious that $\mathcal{K}_i$ are linear operators. Then we complete the proof.
	\end{proof}
\end{Lemma}

\begin{Lemma}\label{lemma_J6}
	When $0<\kappa<1-\mu$, for any $v(t) \in C(I), K_i\in C(I \times I)$, and $K_i(\cdot, s) \in C^{0, \kappa}(I),i=1,2$, there exists a constant $C$ such that
	\begin{align}\label{eq_J6_1}
		\frac{\left|(\mathcal{K}_i v)\left(\theta_1^{\frac{1}{\lambda}}\right)-(\mathcal{K}_i v)\left(\theta_2^{\frac{1}{\lambda}}\right)\right|}{|\theta_1-\theta_2|^\kappa} \leq C \max _{\theta_1 \in I}|v(\theta_1)|, \quad \forall \theta_1, \theta_2 \in I, \theta_1\neq \theta_2 .
	\end{align}
	which is equivalent with 
	\begin{align}\label{eq_J6}
		\left\|(\mathcal{K}_i v)\left(\theta_1^{\frac{1}{\lambda}}\right)\right\|_{0, \kappa} \leq C\|v\|_{\infty}.
	\end{align}
	\begin{proof}
		By virtue of Lemma \ref{lemma_J6org}, we can obtain
		\begin{align}\label{eq_J6_2}
			\frac{\left|(\mathcal{K}_i v)\left(\theta_1^{\frac{1}{\lambda}}\right)-(\mathcal{K}_i v)\left(\theta_2^{\frac{1}{\lambda}}\right)\right|}{|\theta_1^{\frac{1}{\lambda}}-\theta_2^{\frac{1}{\lambda}}|^\kappa} \leq C \max _{\theta_1 \in I}|v(\theta_1)|, \quad \forall \theta_1, \theta_2 \in I, \theta_1\neq \theta_2,
		\end{align}
		
		When $0<\lambda \leq 1$, we have 
		\begin{align}
			\left|\theta_1^{\frac{1}{\lambda}}-\theta_2^{\frac{1}{\lambda}}\right|^\kappa \sim O\left(|\theta_1-\theta_2|^\kappa\right) \text { or }\left|\theta_1^{\frac{1}{\lambda}}-\theta_2^{\frac{1}{\lambda}}\right|^\kappa \sim o\left(|\theta_1-\theta_2|^\kappa\right) \text {,}\notag
		\end{align}
		
		Then it can be derived that
		\begin{align}\label{eq_J6_3}
			\frac{\left|(\mathcal{K}_i v)\left(\theta_1^{\frac{1}{\lambda}}\right)-(\mathcal{K}_i v)\left(\theta_2^{\frac{1}{\lambda}}\right)\right|}{|\theta_1-\theta_2|^\kappa}
			\leq C \frac{\left|(\mathcal{K}_i v)\left(\theta_1^{\frac{1}{\lambda}}\right)-(\mathcal{K}_i v)\left(\theta_2^{\frac{1}{\lambda}}\right)\right|}{|\theta_1^{\frac{1}{\lambda}}-\theta_2^{\frac{1}{\lambda}}|^\kappa}.
		\end{align}
		
		Combing \eqref{eq_J6_1},\eqref{eq_J6_2},\eqref{eq_J6_3},we verify \eqref{eq_J6} and finish the proof.
	\end{proof}
\end{Lemma}

\begin{Lemma}\label{lemma_Hardy}$($$\cite{spectal_TangTao}$ Generalized Hardy's inequality$)$ For all measurable functions $g \geq 0$, weight functions $x$ and $y$, $1 < p \leq q < \infty
	$, the generalized Hardy's inequality can be expressed as follows:
	$$
	\left(\int_a^b\left|\left(\mathcal{M} g\right)(t)\right|^q x(t) d t\right)^{1 / q} \leq C\left(\int_a^b\left|g(x)\right|^p y(t) d t\right)^{1 / p}
	$$ 
if and only if 
	$$
	\sup _{a<t<b}\left(\int_t^b x(s) d s\right)^{1 / q}\left(\int_a^t y^{1-p_{0}}(s) d s\right)^{1 / p_{0}}<\infty, \quad p_{0}=\frac{p}{p-1},
	$$\\
where the operator $\mathcal{M}$ is defined as 
	$$
	(\mathcal{M} g)(t)=\int_a^t \rho(x, \tau) g(\tau) d \tau.
	$$
with $\rho(x, s)$ being a given kernel,  $-\infty \leq a < b \leq \infty$, .
 \end{Lemma}

In order to prove the interpolation error estimate in the $L^{\infty}$-norm and $L_{\omega^{\alpha,\beta,\lambda}}^{2}$-norm, we should introduce the following lemmas:
\begin{Lemma}\label{lemma_Interpolation_infinity}$($see $\cite{Muntz} $$)$
$I_{N, \lambda}^{\alpha, \beta}$ is the interpolation operator of fractional $Jacobi$ polynomials, when $-1<\alpha, \beta \leq-\frac{1}{2}$, for any $0 \leq l \leq m \leq N+1$, we have 
	$$
	\begin{aligned}\label{eq_Interpolation_infinity}
	\left\|v-I_{N, \lambda}^{\alpha, \beta} v\right\|_{\infty} \leq C N^{1 / 2-m}\left\|\partial_\theta^m v\left(\theta^{\frac{1}{\lambda}}\right)\right\|_{0, \omega^{\alpha+m, \beta+m, 1}},\quad \forall v\left(\theta^{\frac{1}{\lambda}}\right) \in B_{\alpha, \beta}^{m, 1}(I), m \geq 1.
	\end{aligned}
	$$ 
\end{Lemma}

\begin{Lemma}\label{lemma_Interpolation_L^2}$($see $\cite{Muntz} $$)$
It holds for $\forall v\left(\theta^{\frac{1}{\lambda}}\right) \in B_{\alpha, \beta}^{m, 1}(I), m \geq 1$, then
	\begin{align}\label{eq_Interpolation_L^2}
		\left\|v-I_{N, \lambda}^{\alpha, \beta} v\right\|_{0,\omega^{\alpha,\beta,1}} \leq C N^{-m}\left\|\partial_\theta^m v\left(\theta^{\frac{1}{\lambda}}\right)\right\|_{0, \omega^{\alpha+m, \beta+m, 1}}.
	\end{align}
\end{Lemma}

\begin{Lemma}\label{lemma_Lesbegue_constant}$($see $\cite{Muntz} $$)$
	Let $\left\{F_{j,\lambda}\left(\theta\right)\right\}_{j=0}^N$ be the generalized Lagrange interpolation basis functions
	associated with the Gauss points of the fractional $Jacobi$ polynomials $J_{N+1}^{\alpha, \beta, \lambda}(x)$. $Lesbegue$ constant  is presented as 
	\begin{align}\label{eq_Lesbegue_constant}
		\left\|I_{N, \lambda}^{\alpha, \beta}\right\|_{\infty}:=\max _{x \in I} \sum_{i=0}^N\left|F_{j,\lambda}\left(\theta\right)\right|= \begin{cases}O(\log N), & -1<\alpha, \beta \leq-\frac{1}{2}, \\ O\left(N^{\gamma+\frac{1}{2}}\right), & \gamma=\max (\alpha, \beta), \text {otherwise}.
		\end{cases}
	\end{align}
\end{Lemma}

\begin{Lemma}\label{lemma_Ii1Ii2}$($see $\cite{Muntz} $$)$
For $\forall v \in B_{\alpha, \beta}^{m, 1}(I), m \geq 1$,$\forall \phi \in \mathcal{P}_N^1(I)$, there are some conclusions as follows
	\begin{align}\label{eq_con&dis}
		\left|(v, \phi)_{\omega^{\alpha, \beta, 1}}-(v, \phi)_{N, \omega^{\alpha, \beta, 1}}\right| \leq C N^{-m}\left\|\partial_\theta^m v\right\|_{0, \omega^{\alpha+m, \beta+m, 1}}\|\phi\|_{0, \omega^{\alpha, \beta, 1}},
	\end{align}
where
	\begin{align}
		&(v, \phi)_{\omega^{\alpha, \beta, 1}}=\int_0^1v(\theta)\phi(\theta)\theta^{\alpha}(1-\theta)^{\beta}d\theta\label{eq_inner_continue},\\
		&(v, \phi)_{N, \omega^{\alpha, \beta, 1}}=\sum_{k=0}^{N}v(\theta_k)\phi(\theta_k)\omega_k.\label{eq_inner_discrete}
	\end{align}
\end{Lemma}

\begin{Lemma}\label{lemma_E6E7_L^2}$($see $\cite{Muntz} $$)$
For any bounded function $v(\theta)$ defined on $I$, there exists a constant $C$ independent of $v$ such that
	\begin{align}\label{eq_J1J2_L^2}
		\qquad\qquad\qquad\qquad\qquad\qquad
		\sup _N\left\|I_{N, \lambda}^{\alpha, \beta} v\right\|_{0,\omega^{\alpha, \beta, \lambda}} \leq C\|v\|_{\infty} .
	\end{align}	
\end{Lemma}

\section{Numerical Scheme} \label{section_Numerical Scheme}
The generalized interpolation operator $I_{N, \lambda}^{\alpha, \beta}$ defined in \eqref{eq_Interpolation_operator_sum} satisfies
	$$
	\begin{aligned}
		I_{N, \lambda}^{\alpha, \beta}v\left(\theta_i\right)=v\left(\theta_i\right),0 \leq i \leq N,\notag
	\end{aligned}
	$$
Rewrite \eqref{eq_prime} and \eqref{eq_origin} through the change of variable $\tau=\varepsilon s$ as follows:
	\begin{align}
		&y^{\prime}(t)= a_1{(t)} y(t)+b_1{(t)} y(\varepsilon t)+f_1{(t)}+\int_0^t(t-s)^{-\mu} K_1(t, s) y(s) d s \notag\\ 
		&\qquad\quad+\varepsilon^{1-\mu}\int_0^t(t-s)^{-\mu}  K_2(t, \varepsilon s) y(\varepsilon s) d s,\label{eq_remark_2}\\
		&y\left(0\right)=y_0.
	\end{align}
After making variable transformations $t=T\theta$ and $s=T\eta$ , \eqref{eq_origin} and \eqref{eq_remark_2} become
	\begin{align}
		&\varphi^{\prime}(\theta)=\tilde{a}_1(\theta) \varphi(\theta)+\tilde{b}_1(\theta)\varphi(\varepsilon \theta)+\tilde{f}_1(\theta)+  \int_0^\theta(\theta-\eta)^{-\mu}  \bar{K}_1(\theta, \eta) \varphi(\eta) d \eta \notag\\
		&\quad\qquad+ \int_0^\theta(\theta-\eta)^{-\mu}\bar{K}_2(\theta, \varepsilon \eta)\varphi(\varepsilon \eta) d \eta\label{eq_phi_prime}\\
		&\varphi(0)=\varphi_0=y_0\label{eq_phi_origin}.
	\end{align}
where
	\begin{align}
		&\tilde{a}_1(\theta)=Ta_1(T\theta),\tilde{b}_1(\theta)=Tb_1(T\theta),\tilde{f}_1(\theta)=Tf_1(T\theta), \varphi(\theta)=y(T\theta),\notag\\
		&\bar{K}_1(\theta, \eta)=T^{2-\mu}K_1(T\theta,T\eta),\bar{K}_2(\theta, \varepsilon \eta)=\varepsilon^{1-\mu}T^{2-\mu}K_2(T\theta, T\varepsilon \eta).\notag
	\end{align}
\eqref{eq_phi_origin} can be transferred into integral form:
\begin{align}
	&\varphi(\theta)=\varphi_0+\int_0^{\theta}\varphi^{\prime}(\eta)d\eta.\label{eq_initial_integral}
\end{align}
\eqref{eq_phi_prime} and \eqref{eq_initial_integral} hold at the collocation points $\lbrace \theta_i \rbrace_{i=0}^{N}$:
\begin{align}
	&\varphi^{\prime}(\theta_i)=\tilde{a}_1(\theta_i) \varphi(\theta_i)+\tilde{f}_1(\theta_i)\varphi(\varepsilon \theta_i)+\tilde{f}_1(\theta_i)+  \int_0^{\theta_i}(\theta_i-\eta)^{-\mu}  \bar{K}_1(\theta_i, \eta) \varphi(\eta) d \eta ,\notag\\
	&\quad\qquad+  \int_0^{\theta_i}(\theta_i-\eta)^{-\mu}\bar{K}_2(\theta_i, \varepsilon \eta)(\varepsilon \eta) d \eta,\label{eq_thetai_collocation_prime}\\
	&\varphi(\theta_i)=\varphi_0+\int_0^{\theta_i}\varphi^{\prime}(\eta)d\eta,\label{eq_thetai_collocation_origin}\\
	&\varphi(\varepsilon\theta_i)=\varphi_0+\varepsilon\int_0^{\theta_i}\varphi^{\prime}(\varepsilon\eta)d\eta.\label{eq_thetai_collocation_eorigin}
\end{align}\\
In order to transfer the integration interval from $[0,\theta_i]$ to $[0,1]$ for analysing easily, we make the variable change $\eta=\theta_i\xi^{\frac{1}{\lambda}}= \eta_i(\xi)$:
	\begin{align}
		&\varphi^{\prime}(\theta_i)=\tilde{a}_1(\theta_i) \varphi(\theta_i)+\tilde{b}_1(\theta_i)\varphi(\varepsilon \theta_i)+\tilde{f}_1(\theta_i)+
		\int_0^1(1-\xi)^{-\mu}\xi^{\frac{1}{\lambda}-1}\widetilde{K}_1(\theta_i, \eta_i(\xi))\varphi( \eta_i(\xi))d\xi\notag\\
		&\quad\qquad+\int_0^1(1-\xi)^{-\mu}\xi^{\frac{1}{\lambda}-1}\widetilde{K}_2(\theta_i,\varepsilon \eta_i(\xi))\varphi(\varepsilon \eta_i(\xi))d\xi,\label{eq_collocation_prime}\\
		&\varphi\left(\theta_i\right)=\varphi_0+\int_0^{\theta_i} \varphi^{\prime}(\eta) d\eta\notag \\
		&\quad\qquad=\varphi_0+\frac{\theta_i}{\lambda} \int_0^1 \xi^{\frac{1}{\lambda}-1} \varphi^{\prime}\left( \eta_i(\xi)\right) d \xi\label{eq_collocation_origin},\\
		&\varphi\left(\varepsilon\theta_i\right)=\varphi_0+\frac{\varepsilon\theta_i}{\lambda} \int_0^1 \xi^{\frac{1}{\lambda}-1} \varphi^{\prime}\left( \varepsilon\eta_i(\xi)\right) d \xi .\label{eq_collocation_eorigin}
	\end{align}\\
where
	\begin{align}
		& \widetilde{K}_1\left(\theta_i,  \eta_i(\xi)\right)
		=\frac{1}{\lambda}\theta_{i}^{1-\mu} \frac{\left(1-\xi^{\frac{1}{\lambda}}\right)^{-\mu}}{(1-\xi)^{-\mu}}\bar{K}_1\left(\theta_i,  \eta_i(\xi)\right),\notag\\
		& \widetilde{K}_2\left(\theta_i, \varepsilon \eta_i(\xi)\right)
		=\frac{1}{\lambda}\theta_{i}^{1-\mu} \frac{\left(1-\xi^{\frac{1}{\lambda}}\right)^{-\mu}}{(1-\xi)^{-\mu}}\bar{K}_2\left(\theta_i,\varepsilon \eta_i(\xi)\right).\notag
	\end{align}
	
By using the $(N+1)$-point fractional $Jacobi-Gauss$ quadrature formula, the two integral terms in \eqref{eq_collocation_prime}. When $\alpha=-\mu,\beta=\frac{1}{\lambda}-1$ are the related parameters, we denote nodes $\lbrace \xi_k\rbrace_{k=0}^{N}$ and the weights $\lbrace \omega_k\rbrace_{k=0}^{N}$.
For \eqref{eq_collocation_origin} and \eqref{eq_collocation_eorigin}, we denote $\lbrace \hat{\xi}_k\rbrace_{k=0}^{N}$ and $\lbrace \hat{\omega}_k\rbrace_{k=0}^{N}$ as fractional $Jacobi-Gauss$ quadrature nodes and weights when $\alpha=0,\beta=\frac{1}{\lambda}-1$, then we will obtain
	\begin{align}
	&\int_0^1(1-\xi)^{-\mu}\xi^{\frac{1}{\lambda}-1}\widetilde{K}_1(\theta_i, \eta_i(\xi))\varphi( \eta_i(\xi))d\xi \approx \sum_{k=0}^N \widetilde{K}_1\left(\theta_i,  \eta_i\left(\xi_k\right)\right) \varphi\left( \eta_i\left(\xi_k\right)\right) \omega_k\label{eq_discrete_K1},\\
	&\int_0^1(1-\xi)^{-\mu}\xi^{\frac{1}{\lambda}-1}\widetilde{K}_2(\theta_i,\varepsilon \eta_i(\xi))\varphi(\varepsilon \eta_i(\xi))d\xi \approx \sum_{k=0}^N \widetilde{K}_2\left(\theta_i, \varepsilon  \eta_i\left(\xi_k\right)\right) \varphi\left(\varepsilon  \eta_i\left(\xi_k\right)\right) \omega_k\label{eq_discrete_K2},\\
	&\frac{\theta_i}{\lambda} \int_0^1 \xi^{\frac{1}{\lambda}-1} \varphi^{\prime}\left( \eta_i(\xi)\right) d \xi\approx
	\sum_{k=0}^{N}\frac{\theta_i}{\lambda}\varphi^{\prime}\left( \eta_i(\hat{\xi}_k)\right)\hat{\omega}_k\label{eq_discrete_origin},\\
	&\frac{\varepsilon\theta_i}{\lambda} \int_0^1 \xi^{\frac{1}{\lambda}-1} \varphi^{\prime}\left(\varepsilon \eta_i(\xi)\right) d \xi\approx
	\sum_{k=0}^{N}\frac{\varepsilon\theta_i}{\lambda}\varphi^{\prime}\left(\varepsilon \eta_i(\hat{\xi}_k)\right)\hat{\omega}_k\label{eq_discrete_eorigin}.
	\end{align}
We denote it as:
	\begin{align}\label{eq_numerical solutions}
		\varphi^{\prime}\left(\theta\right)\approx \varphi_{N}^{\ast}\left(\theta\right)=\sum_{j=0}^{N}\varphi_j^{\ast}F_{j,\lambda}(\theta),
		\varphi\left(\theta\right)\approx \varphi_{N}\left(\theta\right)=\sum_{j=0}^{N}\varphi_jF_{j,\lambda}(\theta).
	\end{align}
where $\varphi_{N}^{\ast}\left(\theta\right)$ is not the exact derive of $\varphi_{N}\left(\theta\right)$. $\varphi\left(\theta\right),\varphi^{\prime}(\theta)$ are approxiated by the generalized $Lagrange$ interpolation polynomials according to \eqref{eq_Interpolation_operator_sum}.

Then we define $\varphi_i,\varphi_i^{\ast}$ and $v_i$ as approxiamations of $\varphi\left(\theta_i\right),\varphi^{\prime}(\theta_i)$ and $\varphi\left(\varepsilon\theta_i\right)$.
The fractional $Jacobi$ collocation method tells us to seek $\lbrace \varphi_i\rbrace_{i=0}^{N},\lbrace \varphi_i^{\ast}\rbrace_{i=0}^{N}$ and $\lbrace v_i\rbrace_{i=0}^{N}$,  holding at the following equations:
	\begin{align}
		&\varphi_i^{\ast}=\tilde{a}_1\left(\theta_i\right) \varphi_i+\tilde{b}_1\left(\theta_i\right) v_i  +\tilde{f}_1(\theta_i)+\sum_{j=0}^N \varphi_j \left(\sum_{k=0}^N \widetilde{K}_1\left(\theta_i,  \eta_i\left(\xi_k\right)\right) F_{j,\lambda}\left( \eta_i\left(\xi_k\right)\right) \omega_k\right)\notag\\
		& \qquad+\sum_{j=0}^N \varphi_j\left(\sum_{k=0}^N \widetilde{K}_2\left(\theta_i, \varepsilon \eta_i\left(\xi_k\right)\right) F_{j,\lambda}\left(\varepsilon \eta_i\left(\xi_k\right)\right) \omega_k\right),\label{eq_collocation*_prime}\\
		&\varphi_i=\varphi_0+\sum_{j=0}^N  \varphi_j^*\left(\sum_{k=0}^N\frac{\theta_i}{\lambda} F_{j,\lambda}\left( \eta_i\left(\hat{\xi}_k\right)\right) \hat{\omega}_k\right),\label{eq_collocation*_origin}\\
		&v_i=\varphi_0+ \sum_{j=0}^N\varphi_j^*\left(\sum_{k=0}^N\frac{\varepsilon\theta_i}{\lambda}F_{j,\lambda}\left(\varepsilon \eta_i\left(\hat{\xi}_k\right)\right) \hat{\omega}_k\right)\label{eq_collocation*_eorigin}.
	\end{align}
We define $U_N^{\ast}=\lbrace\varphi_0^{\ast},\varphi_1^{\ast},\cdots,\varphi_N^{\ast}\rbrace^{T},
U_N=\lbrace\varphi_0,\varphi_1,\cdots,\varphi_N\rbrace^{T},
V_N=\lbrace v_0,v_1,\cdots,v_N\rbrace^{T}
$ so as to obtain the matrix form of \eqref{eq_collocation*_prime}-\eqref{eq_collocation*_eorigin}
	\begin{align}
		&U_N^{\ast}=(A+C+D)U_N+BV_N+F\label{eq_collocation*_matrix_prime},\\
		&U_N=U_0+EU_N^{\ast}\label{eq_collocation*_matrix_origin},\\
		&V_N=U_0+HU_N^{\ast}\label{eq_collocation*_matrix_eorigin}.
	\end{align}
where $A=diag\lbrace\tilde{a}_1\left(\theta_0\right),\tilde{a}_1\left(\theta_1\right),\cdots,\tilde{a}_1\left(\theta_N\right)\rbrace, B=diag\lbrace\tilde{b}_1\left(\theta_0\right),\tilde{b}_1\left(\theta_1\right),\cdots,\tilde{b}_1\left(\theta_N\right)\rbrace,F=\\ \lbrace\tilde{f}_1\left(\theta_0\right),\tilde{f}_1\left(\theta_1\right),\cdots,\tilde{f}_1\left(\theta_N\right)\rbrace^{T},U_0=\lbrace\varphi_0,\varphi_0,\cdots,\varphi_0\rbrace^{T}$. The matrix elements of $C,D,E,H$ are as follows:
	\begin{align}
		&C_{ij}=\sum_{k=0}^N \widetilde{K}_1\left(\theta_i,  \eta_i\left(\xi_k\right)\right) F_{j,\lambda}\left( \eta_i\left(\xi_k\right)\right) \omega_k,\notag\\
		&D_{ij}=\sum_{k=0}^N \widetilde{K}_2\left(\theta_i, \varepsilon \eta_i\left(\xi_k\right)\right) F_{j,\lambda}\left(\varepsilon \eta_i\left(\xi_k\right)\right)\omega_k,\notag\\
		&E_{ij}=\sum_{k=0}^N\frac{\theta_i}{\lambda} F_{j,\lambda}\left( \eta_i\left(\hat{\xi}_k\right)\right) \hat{\omega}_k,\notag\\
		&H_{ij}=\sum_{k=0}^N \frac{\varepsilon\theta_i}{\lambda}F_{j,\lambda}\left(\varepsilon \eta_i\left(\hat{\xi}_k\right)\right) \hat{\omega}_k.\notag
	\end{align}
We can calculate the values of  $\lbrace \varphi_i\rbrace_{i=0}^{N}$ and $\lbrace \varphi_i^{\ast}\rbrace_{i=0}^{N}$ by solving the system of \eqref{eq_collocation*_matrix_prime}-\eqref{eq_collocation*_matrix_eorigin}, then the numerical solutions can be derived through \eqref{eq_numerical solutions} accordingly.
\begin{remark}\label{remark_integral_equal}
	Since $F_{j,\lambda}(\theta)(j=0,1,\cdots,N)$ are fractional Jacobi polynomials of degree not exceeding $N$, $F_{j,\lambda}( \eta_i(\xi))$ are $j$-th Jacobi polynomials with respect to $\xi$ due to the definition of $F_{j,\lambda}( \theta)$. Beseides, we have $\varphi_N\left( \eta_i\left(\xi\right)\right)=\sum_{j=0}^{N}{\varphi_j}^{\ast}F_{j,\lambda}\left( \eta_i\left(\xi\right)\right)\in \mathcal{P}_N^1$. Thus, there is a relationship between the following integral and quadrature formula:
	\begin{align}
		\int_0^{\theta_i} \varphi_{N}^{\ast}(\eta) d\eta
		=&\int_0^{\theta_i}\sum_{j=0}^N\varphi_j^{\ast}F_{j,\lambda}(\eta)d\eta
		=\frac{\theta_i}{\lambda}\int_0^{1}\xi^{\frac{1}{\lambda}-1}\sum_{j=0}^N\varphi_j^{\ast}F_{j,\lambda}\left( \eta_i(\xi)\right)d\xi\notag\\
		=&\sum_{j=0}^N  \varphi_j^*\left(\sum_{k=0}^N\frac{\theta_i}{\lambda} F_{j,\lambda}\left( \eta_i\left(\hat{\xi}_k\right)\right) \hat{\omega}_k\right)
		=\left(\frac{\theta_i}{\lambda}, \varphi_N^{\ast}\left( \eta_i(\cdot)\right)\right) _{N,\omega^{0, \frac{1}{\lambda}-1,1}},\label{eq_integral_equal}
	\end{align}
 We can infer the case with $\varepsilon$ similarly.
\end{remark}

\section{Convergence Analysis}\label{section_Convergence Analysis}
In this section, we will continue to demonstrate the error bounds under $L^{\infty}$ and the weighted  $L^2$-norm of the numerical scheme in Section \ref{section_Numerical Scheme}. Then we will prove the proposed approximations have the rate of exponential convergence,  i.e., the spectral accuracy can be exhibited. \\
Derived from \eqref{eq_inner_continue} and Remark \ref{remark_integral_equal}, rewrite \eqref{eq_collocation_prime}-\eqref{eq_collocation_eorigin} into the form of continuous inner products
		\begin{align}
		&\varphi^{\prime}(\theta_i)=\tilde{a}_1(\theta_i) \varphi(\theta_i)+\tilde{b}_1(\theta_i)\varphi(\varepsilon \theta_i)+\tilde{f}_1(\theta_i)+
		\left(\widetilde{K}_1(\theta_i, \eta_i(\cdot)), \varphi\left( \eta_i(\cdot)\right)\right)_ {\omega^{-\mu,\frac{1}{\lambda}-1,1}}\notag\\
		&\qquad\quad+	\left(\widetilde{K}_2(\theta_i,\varepsilon \eta_i(\cdot)), \varphi\left(\varepsilon \eta_i(\cdot)\right)\right)_ {\omega^{-\mu,\frac{1}{\lambda}-1,1}},\label{eq_inner_continue_prime}\\
		&\varphi\left(\theta_i\right)=\varphi_0+\left(\frac{\theta_i}{\lambda}, \varphi^{\prime}\left( \eta_i(\cdot)\right)\right) _{\omega^{0, \frac{1}{\lambda}-1,1}}\label{eq_inner_continue_origin},\\
		&\varphi\left(\varepsilon\theta_i\right)=\varphi_0+\left(\frac{\varepsilon\theta_i}{\lambda}, \varphi^{\prime}\left(\varepsilon \eta_i(\cdot)\right)\right) _{\omega^{0, \frac{1}{\lambda}-1,1}}.\label{eq_inner_continue_eorigin}
	\end{align}
Based on \eqref{eq_inner_discrete}, \eqref{eq_collocation*_prime}-\eqref{eq_collocation*_eorigin} can be transformed into numerical quadrature inner products
		\begin{align}
		&\varphi_i^{\ast}=\tilde{a}_1\left(\theta_i\right) \varphi_i+\tilde{b}_1\left(\theta_i\right) v_i  +\tilde{f}_1(\theta_i)+	\left(\widetilde{K}_1(\theta_i, \eta_i(\cdot)), \varphi_N\left( \eta_i(\cdot)\right)\right)_ {N,\omega^{-\mu,\frac{1}{\lambda}-1,1}}\notag\\
		&\qquad+\left(\widetilde{K}_2(\theta_i,\varepsilon \eta_i(\cdot)), \varphi_N\left(\varepsilon \eta_i(\cdot)\right)\right)_ {N,\omega^{-\mu,\frac{1}{\lambda}-1,1}},\label{eq_inner_discrete_prime}\\
		&\varphi_i=\varphi_0+\left(\frac{\theta_i}{\lambda}, \varphi_N^{\ast}\left( \eta_i(\cdot)\right)\right) _{N,\omega^{0, \frac{1}{\lambda}-1,1}},\label{eq_inner_discrete_origin}\\
		&v_i=\varphi_0+ \left(\frac{\varepsilon\theta_i}{\lambda}, \varphi_{N}^{\ast}\left(\varepsilon \eta_i(\cdot)\right)\right) _{N,\omega^{0, \frac{1}{\lambda}-1,1}} .\label{eq_inner_discrete_eorigin}
	\end{align}
Add on both sides of the equation \eqref{eq_inner_discrete_prime} by 
	\begin{align}
	\int_0^{\theta_i}(\theta_i-\eta)^{-\mu}\bar{K}_1(\theta_i, \eta) \varphi_N(\eta) d \eta=\left(\widetilde{K}_1(\theta_i, \eta_i(\cdot)), \varphi_N\left( \eta_i(\cdot)\right)\right)_ {\omega^{-\mu,\frac{1}{\lambda}-1,1}},
	\end{align}
	and 
	\begin{align}
		 \int_0^{\theta_i}(\theta_i-\eta)^{-\mu}\bar{K}_2(\theta_i, \varepsilon \eta)\varphi_N(\varepsilon \eta) d \eta=\left(\widetilde{K}_2(\theta_i,\varepsilon \eta_i(\cdot)), \varphi_N\left(\varepsilon \eta_i(\cdot)\right)\right)_ {\omega^{-\mu,\frac{1}{\lambda}-1,1}},
	\end{align}
as previously derived.
By virtue of Remark \ref{remark_integral_equal}, \eqref{eq_inner_discrete_prime}, \eqref{eq_inner_discrete_origin} and \eqref{eq_inner_discrete_eorigin} can be turned into
		\begin{align}
		&\varphi_i^{\ast}=\tilde{a}_1\left(\theta_i\right) \varphi_i+\tilde{b}_1\left(\theta_i\right) v_i
		+\tilde{f}_1(\theta_i)
		+  \int_0^{\theta_i}(\theta_i-\eta)^{-\mu}  \bar{K}_1(\theta_i, \eta) \varphi_N(\eta) d \eta ,\notag\\
		&\qquad+  \int_0^{\theta_i}(\theta_i-\eta)^{-\mu}  \bar{K}_2(\theta_i, \varepsilon \eta)\varphi_N(\varepsilon \eta) d \eta-I_{i,1}-I_{i,2},\label{eq_thetai_numerical_prime}\\
		&\varphi_i=\varphi_0+\int_0^{\theta_i} \varphi_{N}^{\ast}(\eta) d\eta,\label{eq_thetai_numerical_origin}\\
		&v_i=\varphi_0+\varepsilon\int_0^{\theta_i} \varphi_{N}^{\ast}(\varepsilon\eta) d\eta.\label{eq_thetai_numerical_eorigin}
	\end{align}
where we adopt this mathematical notation:
	\begin{align}
		&I_{i,1}=\left(\widetilde{K}_1(\theta_i, \eta_i(\cdot)), \varphi_N\left( \eta_i(\cdot)\right)\right)_ {\omega^{-\mu,\frac{1}{\lambda}-1,1}}-\left(\widetilde{K}_1(\theta_i, \eta_i(\cdot)), \varphi_N\left( \eta_i(\cdot)\right)\right)_ {N,\omega^{-\mu,\frac{1}{\lambda}-1,1}},\label{eq_Ii1}\\
		&I_{i,2}=\left(\widetilde{K}_2(\theta_i,\varepsilon \eta_i(\cdot)), \varphi_N\left(\varepsilon \eta_i(\cdot)\right)\right)_ {\omega^{-\mu,\frac{1}{\lambda}-1,1}}-	\left(\widetilde{K}_2(\theta_i,\varepsilon \eta_i(\cdot)), \varphi_N\left(\varepsilon \eta_i(\cdot)\right)\right)_ {N,\omega^{-\mu,\frac{1}{\lambda}-1,1}}.\label{eq_Ii2}
	\end{align}
When we apply Lemma \ref{lemma_Ii1Ii2}, the bound of their absolute value are
	\begin{align}
		& \left|I_{i, 1}\right| \leq C N^{-m}\left\|\partial_\theta^m \widetilde{K}_1\left(\theta_i,  \eta_i(\cdot)\right)\right\|_ {0,\omega^{m-\mu,{m+\frac{1}{\lambda}}-1,1}}
		\left\|\varphi_N\left( \eta_i\left(\cdot\right)\right)\right\|_ {0,\omega^{-\mu,\frac{1}{\lambda}-1,1}},\label{eq_E1_Ii1_absolute}\\
		& \left|I_{i, 2}\right| \leq C N^{-m}\left\|\partial_\theta^m \widetilde{K}_2\left(\theta_i, \varepsilon \eta_i(\cdot)\right)\right\|_ {0,\omega^{m-\mu,{m+\frac{1}{\lambda}}-1,1}}
		\left\|\varphi_N\left(\varepsilon \eta_i\left(\cdot\right)\right)\right\|_ {0,\omega^{-\mu,\frac{1}{\lambda}-1,1}}.\label{eq_E2_Ii2_absolute}
	\end{align}
After subtracting \eqref{eq_thetai_numerical_prime} from \eqref{eq_thetai_collocation_prime},  \eqref{eq_thetai_numerical_origin} from \eqref{eq_thetai_collocation_origin}, and \eqref{eq_thetai_numerical_eorigin} from \eqref{eq_thetai_collocation_eorigin},  we can deduce the equations for the errors:
	\begin{align}
		&\varphi^{\prime}(\theta_i)-\varphi_i^{\ast}=\tilde{a}_1(\theta_i) \left(\varphi(\theta_i)-\varphi_i\right)+\tilde{b}_1(\theta_i)\left(\varphi(\varepsilon \theta_i)-v_i\right)
		+  \int_0^{\theta_i}(\theta_i-\eta)^{-\mu}  \bar{K}_1(\theta_i, \eta) e(\eta) d \eta\notag\\
		&\qquad\qquad\qquad+  \int_0^{\theta_i}(\theta_i-\eta)^{-\mu}  \bar{K}_2(\theta_i, \varepsilon \eta)e(\varepsilon \eta) d \eta+I_{i,1}+I_{i,2},\label{eq_subtraction1_prime}\\
		&\varphi(\theta_i)-\varphi_i=\int_0^{\theta_i} e^{\ast}(\eta) d\eta,\label{eq_subtraction1_origin}\\
		&\varphi(\varepsilon \theta_i)-v_i=\varepsilon\int_0^{\theta_i} e^{\ast}(\varepsilon\eta) d\eta.\label{eq_subtraction1_eorigin}
	\end{align}
where $e(\theta)=\varphi(\theta)-\varphi_N(\theta),e^{\ast}(\theta)=\varphi^{\prime}(\theta)-\varphi_N^{\ast}(\theta)$. Substituting \eqref{eq_subtraction1_origin} and \eqref{eq_subtraction1_eorigin} into \eqref{eq_subtraction1_prime} can yield:
	\begin{align}
		\varphi^{\prime}(\theta_i)-\varphi_i^{\ast}=&\tilde{a}_1(\theta_i) \int_0^{\theta_i} e^{\ast}(\eta) d\eta+\varepsilon\tilde{b}_1(\theta_i)\int_0^{\theta_i} e^{\ast}(\varepsilon\eta) d\eta\notag\\
		&+  \int_0^{\theta_i}(\theta_i-\eta)^{-\mu}  \bar{K}_1(\theta_i, \eta) e(\eta) d \eta
		+  \int_0^{\theta_i}(\theta_i-\eta)^{-\mu}  \bar{K}_2(\theta_i, \varepsilon \eta)e(\varepsilon \eta) d \eta+I_{i,1}+I_{i,2}.\label{eq_subtraction2_prime}
	\end{align}
After multiplying $F_{i,\lambda}(\theta)$ on both sides of \eqref{eq_subtraction2_prime}, \eqref{eq_subtraction1_origin} and \eqref{eq_subtraction1_eorigin}, and summing up from $i=0$ to $N$, \eqref{eq_subtraction2_prime} will become
	\begin{align}
		I_{N, \lambda}^{\alpha, \beta} \varphi^{\prime}(\theta)-\varphi_N^*(\theta)=
		& I_{N, \lambda}^{\alpha, \beta}\left(\tilde{a}_1\left(\theta\right) \int_0^\theta e^*(\eta) d \eta\right)
		+ \varepsilon I_{N, \lambda}^{\alpha, \beta}\left(\tilde{b}_1\left(\theta\right) \int_0^\theta e^*(\varepsilon \eta) d \eta\right)\notag\\
		& +I_{N, \lambda}^{\alpha, \beta}\left(  \int_0^\theta\left(\theta-\eta\right)^{-\mu}  \bar{K}_1(\theta, \eta) e(\eta) d \eta\right) \notag\\
		&+I_{N, \lambda}^{\alpha, \beta}\left(  \int_0^\theta\left(\theta-\eta\right)^{-\mu}  \bar{K}_2(\theta,\varepsilon \eta)e(\varepsilon \eta) d \eta\right) \notag\\
		&+ \sum_{i=0}^N I_{i, 1} F_{i,\lambda}(\theta)+\sum_{i=0}^N I_{i, 2} F_{i,\lambda}(\theta)\label{eq_sum1_prime}\\
		I_{N, \lambda}^{\alpha \beta} \varphi(\theta)-\varphi_N(\theta)= & I_{N, \lambda}^{\alpha, \beta}\left(\int_0^\theta e^*(\eta) d \eta\right)\label{eq_sum1_origin},\\
		I_{N, \lambda}^{\alpha \beta} \varphi(\varepsilon\theta)-\varphi_N(\varepsilon\theta)= & I_{N, \lambda}^{\alpha, \beta}\left(\varepsilon\int_0^\theta e^*(\varepsilon\eta) d \eta\right)\label{eq_sum1_eorigin}.
	\end{align}
Adding and subtracting $\varphi^{\prime}(\theta)$ to left side of \eqref{eq_sum1_prime}, $\varphi(\theta)$ to left side of \eqref{eq_sum1_origin} and $\varphi(\varepsilon\theta)$ to left side of \eqref{eq_sum1_eorigin} respectively yield the following results:
	\begin{align}
		e^{\ast}(\theta)=
		&\tilde{a}_1\left(\theta\right) \int_0^\theta e^*(\eta) d \eta
		+ \varepsilon \tilde{b}_1\left(\theta\right) \int_0^\theta e^*(\varepsilon \eta) d \eta+  \int_0^\theta\left(\theta-\eta\right)^{-\mu}  \bar{K}_1(\theta, \eta) e(\eta) d \eta \notag\\
		&+  \int_0^\theta\left(\theta-\eta\right)^{-\mu}  \bar{K}_2(\theta,\varepsilon \eta)e(\varepsilon \eta) d\eta
		+\sum_{p=1}^{7}E_p(\theta),\label{eq_sum*_prime}\\
		e(\theta)=&\int_0^\theta e^*(\eta) d \eta+E_8(\theta)+E_9(\theta),\label{eq_sum*_origin}\\
		e(\varepsilon\theta)=&\varepsilon\int_0^\theta e^*(\varepsilon\eta) d \eta+E_8(\varepsilon\theta)+E_9(\varepsilon\theta).\label{eq_sum*_eorigin}
	\end{align}
where
	\begin{align}
		&E_1(\theta)=\sum_{i=0}^N I_{i, 1} F_{i,\lambda}(\theta),  
		E_2(\theta)=\sum_{i=0}^N I_{i, 2} F_{i,\lambda}(\theta),
		E_3(\theta)= \varphi^{\prime}(\theta)-I_{N, \lambda}^{\alpha, \beta} \varphi^{\prime}(\theta),\notag\\
		&E_4(\theta)= I_{N, \lambda}^{\alpha, \beta}\left(\tilde{a}_1\left(\theta\right) \int_0^\theta e^*(\eta) d \eta\right)-\tilde{a}_1\left(\theta\right) \int_0^\theta e^*(\eta) d \eta,\notag\\
		&E_5(\theta)=\varepsilon I_{N, \lambda}^{\alpha, \beta}\left(\tilde{b}_1\left(\theta\right) \int_0^\theta e^*(\varepsilon\eta) d \eta\right)- \varepsilon \tilde{b}_1\left(\theta\right) \int_0^\theta e^*(\varepsilon \eta) d \eta,\notag\\
	\end{align}
	\begin{align}
		&E_6(\theta)=I_{N, \lambda}^{\alpha, \beta}\left(  \int_0^\theta\left(\theta-\eta\right)^{-\mu}  \bar{K}_1(\theta, \eta) e(\eta) d \eta\right)
		-  \int_0^\theta\left(\theta-\eta\right)^{-\mu}  \bar{K}_1(\theta, \eta) e(\eta) d \eta ,\notag\\
		&E_7(\theta)=I_{N, \lambda}^{\alpha, \beta}\left(  \int_0^\theta\left(\theta-\eta\right)^{-\mu}  \bar{K}_2(\theta,\varepsilon \eta)e(\varepsilon \eta) d \eta\right)
		-  \int_0^\theta\left(\theta-\eta\right)^{-\mu}  \bar{K}_2(\theta,\varepsilon \eta)e(\varepsilon \eta) d\eta,\notag\\
		&E_8(\theta)=\varphi(\theta)-I_{N, \lambda}^{\alpha \beta} \varphi(\theta),
		E_9(\theta)= I_{N, \lambda}^{\alpha, \beta}\left(\int_0^\theta e^*(\eta) d \eta\right)-\int_0^\theta e^*(\eta) d \eta.\notag
	\end{align}
\begin{theorem}\label{theorem_Gronwall}
The  error of $e^{\ast}(\theta)$ and $e(\theta)$ can be bounded by $E_p(\theta)(p=1,2,\cdots,9)$, the specific forms are as follows:
	\begin{align}
		\|{e}^{*}(\theta)\|_{\infty} \leq C\sum_{p=1}^9\left\|E_p(\theta)\right\|_{\infty},\label{theorem_normcontrol_e*_infty}\\
		\|e(\theta)\|_{\infty} \leq C \sum_{p=1}^9\left\|E_p(\theta)\right\|_{\infty}.\label{theorem_normcontrol_e_infty}
	\end{align}
	\begin{align}
		\left\|e^*(\theta)\right\|_{0,\omega^{\alpha,\beta,1}} 
		&\leq C\left(\sum_{p=1}^{9}\left\|E_p(\theta)\right\|_{0,\omega^{\alpha,\beta,1}}+\left\|E_8(\theta)\right\|_{\infty}+\left\|E_9(\theta)\right\|_{\infty} \right)\label{theorem_normcontrol_e*_w}\\
		\left\|e(\theta)\right\|_{0,\omega^{\alpha,\beta,1}}  &\leq C
		\left(\left\|e^*(\theta)\right\|_{0,\omega^{\alpha,\beta,1}}+\left\|E_8(\theta)\right\|_{0,\omega^{\alpha,\beta,1}}+\left\|E_9(\theta)\right\|_{0,\omega^{\alpha,\beta,1}} \right)\label{theorem_normcontrol_e_w}
	\end{align}
where $e(\theta)=\varphi(\theta)-\varphi_N(\theta),e^{\ast}(\theta)=\varphi^{\prime}(\theta)-\varphi_N^{\ast}(\theta)$. 
\end{theorem}
\begin{proof}
After substituting \eqref{eq_sum*_origin} and \eqref{eq_sum*_eorigin} into \eqref{eq_sum*_prime}, we can obtain an inequality containing the errors $e^{\ast}(\theta)$ and $e(\theta)$:
	\begin{align}
		e^*(\theta)= & \tilde{a}_1(\theta) \int_0^\theta e^*(\eta) d \eta+e \tilde{b}_1(\theta) \int_0^\theta e^*(\varepsilon \eta) d \eta   +\sum_{p=1}^7 E_p(\theta) \notag\\
		& +  \int_0^\theta(\theta-\eta)^{-\mu}  \bar{K}_1(\theta, \eta) e(\eta) d \eta+  \int_0^\theta(\theta-\eta)^{-\mu}  \bar{K}_2(\theta, \varepsilon \eta) e(\varepsilon \eta) d \eta\notag \\
		= & \tilde{a}_1(\theta) \int_0^\theta e^*(\eta) d \eta+\tilde{b}_1(\theta) \int_0^{\varepsilon \theta} e^*(\eta) d \eta+J(\theta)\notag \\
		& +  \int_0^\theta(\theta-\eta)^{-\mu}  \bar{K}_1(\theta, \eta)\left(\int_0^\eta e^*(\sigma)\right) d \eta\
		+\int_0^\theta(\theta-\eta)^{-\mu}   \bar{K}_2(\theta,\varepsilon\eta)\left(\int_0^{\varepsilon \sigma} e^*(\sigma) d \sigma\right)d\eta\notag \\
		= & \tilde{a}_1(\theta) \int_0^\theta e^*(\eta) d \eta+\tilde{b}_1(\theta) \int_0^{\varepsilon \theta} e^*(\eta) d \eta+J(\theta)\notag \\
		& +  \int_0^\theta\left(\int_\eta^\theta(\theta-\sigma)^{-\mu}  \bar{K}_1(\theta, \sigma) d \sigma\right) e^*(\eta) d \eta 
		+  \int_0^{\varepsilon \theta}\left(\int_{\frac{\eta}{\varepsilon}}^\theta(\theta-\sigma)^{-\mu}  \bar{K}_2(\theta, \varepsilon\sigma) d \sigma\right) e^*(\eta) d \eta\notag\\
		=& J(\theta)+\int_0^\theta\left(\int_\eta^\theta(\theta-\sigma)^{-\mu}  \bar{K}_1(\theta, \sigma) d \sigma+ \tilde{a}_1(\theta)\right) e^*(\eta) d \eta \notag\\
		& +  \int_0^{\varepsilon \theta}\left(\int_{\frac{\eta}{\varepsilon}}^\theta(\theta-\sigma)^{-\mu}  \bar{K}_2(\theta, \varepsilon\sigma) d \sigma+ \tilde{b}_1(\theta)\right) e^*(\eta) d \eta.\label{eq_Gronwall_origin}
	\end{align}
where
	\begin{align}
		J(\theta)=&  \int_0^\theta(\theta-\eta)^{-\mu}  \bar{K}_1(\theta, \eta)\left(E_8(\eta)+E_9(\eta)\right) d \eta\notag\\
		&+  \int_0^\theta(\theta-\eta)^{-\mu}  \bar{K}_2(\theta,\varepsilon\eta)\left(E_8(\varepsilon\eta)+E_9(\varepsilon\eta)\right) d\eta
		+\sum_{p=1}^7 E_p(\theta).\label{eq_Gronwall_J}
	\end{align}
Define the area $D:\lbrace(\theta,\sigma):0 \leq \sigma \leq \theta,\theta \in [0,1]\rbrace$,then denote $\underset{(\theta,\sigma) \in D}{\max}\left|\bar{K}_1(\theta,\sigma)\right|=\bar{K}_1^{max}$, 
$\underset{(\theta,\sigma) \in D}{\max}\left|\bar{K}_2(\theta,\varepsilon\sigma)\right|=\bar{K}_2^{max}$, 
$\underset{\theta\in[0,1]}{\max}\left|\tilde{a}_1(\theta)\right|=a^{max}$, $\underset{\theta\in[0,1]}{\max}\left|\tilde{b}_1(\theta)\right|=b^{max}$, we have:
	\begin{align}
		&\left| \left(\int_\eta^\theta(\theta-\sigma)^{-\mu}  \bar{K}_1(\theta, \sigma) d \sigma+ \tilde{a}_1(\theta)\right)\right| \notag\\
		\leq &\left| \left(\int_0^\theta(\theta-\sigma)^{-\mu}  \bar{K}_1(\theta, \sigma) d \sigma+ \tilde{a}_1(\theta)\right)\right| \notag\\
		\leq& \left|B(1-\mu,1)\bar{K}_1^{max}+\tilde{a}_1(\theta)\right|\notag\\
		\leq&\left|B(1-\mu,1)\bar{K}_1^{max}\right|+a^{max}\notag\\
		\equiv &C_1.\label{eq_Gronwall_C1}
	\end{align}
	\begin{align}
		&\left| \left(\int_{\frac{\eta}{\varepsilon}}^\theta(\theta-\sigma)^{-\mu}  \bar{K}_2(\theta, \varepsilon\sigma) d \sigma+ \tilde{b}_1(\theta)\right)\right|\notag\\
		\leq&\bar{K}_2^{max}\left| \left(\int_{\frac{\eta}{\varepsilon}}^\theta(\theta-\sigma)^{-\mu}d \sigma\right)\right|+b^{max}\notag\\
		\leq&\bar{K}_2^{max}B(1-\mu, 1)+b^{max} \notag\\
		\equiv&C_2.\label{eq_Gronwall_C2}
	\end{align}
where $0 \leq \frac{\eta}{\varepsilon} \leq \theta$.
Taking the absolute value of both sides of \eqref{eq_Gronwall_origin}. By the triangle inequality and \eqref{eq_Gronwall_C1} and \eqref{eq_Gronwall_C2}, we can obtain
	\begin{align}
		\left|e^*(\theta)\right| & \leq C_1 \int_0^\theta\left|e^*(\eta)\right| d \eta+C_2 \int_0^{\varepsilon \theta}\left|e^*(\eta)\right| d \eta+|J(\theta)| \notag\\ 
		& \leq\left(C_1+C_2\right) \int_0^\theta\left|e^*(\eta)\right| d \eta+|J(\theta)|.\label{eq_Gronwall_e*}
	\end{align}
Lemma \ref{lemma_Chen_Lemma3.5_norm} tells us
	\begin{align}
		\|{e}^{*}(\theta)\|_{\infty} \leq C\|J(\theta)\|_{\infty}.
	\end{align}
It follows from the relationship between $e(\theta)$ and $e^*(\theta)$ in \eqref{eq_sum*_origin}, we can infer that
	\begin{align}
		\|e(\theta)\|_{\infty} \leq \|{e}^{*}(\theta)\|_{\infty}+\left\|E_8(\theta)\right\|_{\infty}+\left\|E_9(\theta)\right\|_{\infty} .\label{e*_e}
	\end{align}
It can be determined from \eqref{eq_Gronwall_J} that
	\begin{align}
		\|J(\theta)\|_{\infty} \leq C \sum_{p=1}^9\left\|E_p(\theta)\right\|_{\infty}.
	\end{align}
In conclusion,
	\begin{align}
		&\|{e}^{*}(\theta)\|_{\infty} \leq C\sum_{p=1}^9\left\|E_p(\theta)\right\|_{\infty},\\
		&\|e(\theta)\|_{\infty} \leq C \sum_{p=1}^9\left(\left\|E_p(\theta)\right\|_{\infty}+\left\|E_8(\theta)\right\|_{\infty}+\left\|E_9(\theta)\right\|_{\infty}\right).
	\end{align}
By virtue of \eqref{eq_sum*_origin}, \eqref{eq_Gronwall_J}, \eqref{eq_Gronwall_e*}, Lemma \ref{lemma_Chen_Lemma3.5_norm}, and Lemma \ref{lemma_Hardy}, we can derive the following estimates
	\begin{align}
		\left\|e^*(\theta)\right\|_{0,\omega^{\alpha,\beta,1}}  &\leq C\left( \sum_{p=1}^{9}\left\|E_p(\theta)\right\|_{0,\omega^{\alpha,\beta,1}}
		+\left\|E_8(\varepsilon\eta)\right\|_{0,\omega^{\alpha,\beta,1}}+\left\|E_8(\varepsilon\eta)\right\|_{0,\omega^{\alpha,\beta,1}}\right)\notag\\
		&\leq C\left(\sum_{p=1}^{9}\left\|E_p(\theta)\right\|_{0,\omega^{\alpha,\beta,1}}+\left\|E_8(\theta)\right\|_{\infty}+\left\|E_9(\theta)\right\|_{\infty} \right)\\
		\left\|e(\theta)\right\|_{0,\omega^{\alpha,\beta,1}}  &\leq C
		\left(\left\|e^*(\theta)\right\|_{0,\omega^{\alpha,\beta,1}}+\left\|E_8(\theta)\right\|_{0,\omega^{\alpha,\beta,1}}+\left\|E_9(\theta)\right\|_{0,\omega^{\alpha,\beta,1}} \right).\label{eq_e_e*}
	\end{align}
\end{proof}
Next, we will demonstrate convergence analysis in $L^{\infty}$-norm.

\subsection{Error estimate in $L^{\infty}$-norm}
\begin{theorem}\label{theorem_infty_norm}
	$\varphi(\theta)$ is the exact solution of the VIDEs \eqref{eq_phi_prime} and \eqref{eq_phi_origin} with sufficiently smooth.	
	$$
	\begin{aligned}
		\varphi_{N}^{\ast}\left(\theta\right)=\sum_{j=0}^{N}\varphi_j^{\ast}F_{j,\lambda}(\theta),
		\varphi_{N}\left(\theta\right)=\sum_{j=0}^{N}\varphi_jF_{j,\lambda}(\theta).\notag
	\end{aligned}
	$$
	$\lbrace\varphi_j^*\rbrace_{j=0}^N,\lbrace\varphi_j\rbrace_{j=0}^N$ have been defined in \eqref{eq_collocation*_prime} and \eqref{eq_collocation*_origin}.
	If $ \varphi\left(\theta^{\frac{1}{\lambda}}\right) \in B_{\alpha, \beta}^{m, 1}(I), \partial_\theta \varphi\left(\theta^{\frac{1}{\lambda}}\right) \in B_{\alpha, \beta}^{m, 1}(I), \\
	\tilde{a}_1(\theta), \tilde{b}_1(\theta), \tilde{f}_1(\theta) \in C^m(I), \bar{K}_1(\theta, \eta), \bar{K}_2(\theta, \varepsilon \eta) \in C^m(I \times I),\text{ where }m \geq 1$. When $-1<\alpha, \beta \leq-\frac{1}{2},$$0<\mu<1,$ where $\mu$ is a real number related with the weakly singular kernel and $0 \leq \kappa \leq 1-\mu,$
	then we can obtain:
	\begin{align}
	\left\|e^*(\theta)\right\|_{\infty}\leq &CN^{-m}\Bigl\{\log N\bigl(\mathcal{K}^*\left\|\varphi(\theta)\right\|_{\infty}
	+N^{\frac{1}{2}-\kappa}\left\|\partial_\theta^m \varphi\left(\theta^{\frac{1}{\lambda}}\right)\right\|_{0, \omega^{\alpha+m,\beta+m,1}}\bigr)+N^{\frac{1}{2}}\Phi\Bigr\},\label{theorem_e*_infty}\\
	\left\|e(\theta)\right\|_{\infty}\leq&CN^{-m}\Bigl\{\log 	N\bigl(\mathcal{K}^*\left\|\varphi(\theta)\right\|_{\infty}
	+N^{\frac{1}{2}-\kappa}\left\|\partial_\theta^m 	\varphi\left(\theta^{\frac{1}{\lambda}}\right)\right\|_{0, \omega^{\alpha+m,\beta+m,1}}\bigr)
	+N^{\frac{1}{2}}\Phi\Bigr\}.\label{theorem_e_infty}
	\end{align}
	where $N$ is a sufficiently large positive integer and
	\begin{align}
		\mathcal{K}^*=&\max\limits_{0\leq i\leq N}\left\|\partial_\theta^m \widetilde{K}_1\left(\theta_i, \eta_i(\cdot)\right)\right\|_ {0,\omega^{m-\mu,{m+\frac{1}{\lambda}}-1,1}}+\max\limits_{0\leq i\leq N}\left\|\partial_\theta^m \widetilde{K}_2\left(\theta_i, \varepsilon \eta_i(\cdot)\right)\right\|_ {0,\omega^{m-\mu,{m+\frac{1}{\lambda}}-1,1}},\notag\\
		\Phi=&\left\|\partial_\theta^{m+1} \varphi\left(\theta^{\frac{1}{\lambda}}\right)\right\|_{0, \omega^{\alpha+m,\beta+m,1}}+\left\|\partial_\theta^m \varphi\left(\theta^{\frac{1}{\lambda}}\right)\right\|_{0, \omega^{\alpha+m,\beta+m,1}}.\notag
	\end{align}
\end{theorem}
\begin{proof}
As discussed in \eqref{eq_Ii1},\eqref{eq_Ii2},\eqref{e*_e}, Lemma \ref{lemma_Lesbegue_constant} and Lemma \ref{lemma_Ii1Ii2}, we can deduce:
	\begin{align}
		\left\|E_1(\theta)\right\|_{\infty}
		&\leq C\max\limits_{0\leq i\leq N}\left|I_{i, 1}\right| \max\limits_{\theta \in I}\sum_{j=0}^{N}\left|F_{j,\lambda}(\theta)\right|\notag\\
		&\leq C N^{-m}\log N\max\limits_{0\leq i\leq N}\left\|\partial_\theta^m \widetilde{K}_1\left(\theta_i,  \eta_i(\cdot)\right)\right\|_ {0,\omega^{m-\mu,{m+\frac{1}{\lambda}}-1,1}}
		\left\|\varphi_N\left( \eta_i\left(\cdot\right)\right)\right\|_{0,\omega^{-\mu,\frac{1}{\lambda}-1,1}},\notag\\
		&\leq C N^{-m}\log N\max\limits_{0\leq i\leq N}\left\|\partial_\theta^m \widetilde{K}_1\left(\theta_i,  \eta_i(\cdot)\right)\right\|_ {0,\omega^{m-\mu,{m+\frac{1}{\lambda}}-1,1}}
		\left(\left\|\varphi(\theta)\right\|_{\infty}+\left\|e(\theta)\right\|_\infty\right)\notag\\
		&\leq C N^{-m}\log N\max\limits_{0\leq i\leq N}\left\|\partial_\theta^m \widetilde{K}_1\left(\theta_i,  \eta_i(\cdot)\right)\right\|_ {0,\omega^{m-\mu,{m+\frac{1}{\lambda}}-1,1}}\notag\\
		&\quad\left(\left\|\varphi(\theta)\right\|_{\infty}+\|{e}^{*}(\theta)\|_{\infty}+\left\|E_8(\theta)\right\|_{\infty}+\left\|E_9(\theta)\right\|_{\infty}\right),\label{eq_E1_infty}
	\end{align}
	\begin{align}
		\left\|E_2(\theta)\right\|_{\infty}
		&\leq C\max\limits_{0\leq i\leq N}\left|I_{i, 2}\right| \max\limits_{(\theta \in I}\sum_{j=0}^{N}\left|F_{j,\lambda}(\theta)\right|\notag\\
		&\leq C N^{-m}\log N\max\limits_{0\leq i\leq N}\left\|\partial_\theta^m \widetilde{K}_2\left(\theta_i, \varepsilon \eta_i(\cdot)\right)\right\|_ {0,\omega^{m-\mu,{m+\frac{1}{\lambda}}-1,1}}
		\left\|\varphi_N\left(\varepsilon \eta_i\left(\cdot\right)\right)\right\|_ {0,\omega^{-\mu,\frac{1}{\lambda}-1,1}}\notag\\
		&\leq C N^{-m}\log N\max\limits_{0\leq i\leq N}\left\|\partial_\theta^m \widetilde{K}_2\left(\theta_i, \varepsilon \eta_i(\cdot)\right)\right\|_ {0,\omega^{m-\mu,{m+\frac{1}{\lambda}}-1,1}}
		\left(\left\|\varphi(\theta)\right\|_{\infty}+\left\|e(\theta)\right\|_{\infty}\right)\notag\\
		&\leq C N^{-m}\log N\max\limits_{0\leq i\leq N}\left\|\partial_\theta^m \widetilde{K}_2\left(\theta_i, \varepsilon \eta_i(\cdot)\right)\right\|_ {0,\omega^{m-\mu,{m+\frac{1}{\lambda}}-1,1}}\notag\\
		&\quad\left(\left\|\varphi(\theta)\right\|_{\infty}+\|{e}^{*}(\theta)\|_{\infty}+\left\|E_8(\theta)\right\|_{\infty}+\left\|E_9(\theta)\right\|_{\infty}\right).\label{eq_E2_infty}
	\end{align}
Applying Lemma \ref{lemma_Interpolation_infinity}, the error estimate will be derived:
	\begin{align}
		&\left\|E_3(\theta)\right\|_{\infty} \leq C N^{\frac{1}{2}-m}\left\|\partial_\theta^{m+1} \varphi\left(\theta^{\frac{1}{\lambda}}\right)\right\|_{0, \omega^{\alpha+m,\beta+m,1}},\label{eq_E3_infity}\\
		&\left\|E_8(\theta)\right\|_{\infty}\leq C N^{\frac{1}{2}-m}\left\|\partial_\theta^m \varphi\left(\theta^{\frac{1}{\lambda}}\right)\right\|_{0, \omega^{\alpha+m,\beta+m,1}}.\label{eq_E8_infty}
	\end{align}

$\int_0^{\theta^{\frac{1}{\lambda}}} e^*(\eta) d \eta=\int_0^{\theta^{\frac{1}{\lambda}}}\left(\varphi^{\prime}(\eta)-\varphi_N^*(\eta)\right) d \eta,$
we set $ \varphi\left(\theta^{\frac{1}{\lambda}}\right) \in B_{\alpha, \beta}^{1,1}(I) \text { , then } \int_0^{\theta^{\frac{1}{\lambda}}} \varphi^{\prime}(\eta) d \eta \in B_{\alpha, \beta}^{1,1}(I)$. 
We have $\int_0^{\theta^{\frac{1}{\lambda}}} \varphi_N^*(\eta) d \eta \in B_{\alpha, \beta}^{1,1}(I)$ from  $\varphi_N^*\left(\theta^{\frac{1}{\lambda}}\right) \in B_{\alpha, \beta}^{1,1}(I)$, then $\left\|E_4(\theta)\right\|_{\infty}$ can be estimated by Lemma \ref{lemma_Interpolation_infinity} with the case $m=1$,
		\begin{align}
		\left\|E_4(\theta)\right\|_{\infty}
		& \leq\left\|\left(I_{N, \lambda}^{a, \beta}-I\right) \tilde{a}_1(\theta) \int_0^\theta e^*(\eta) d \eta\right\|_{\infty}\notag \\
		& \leq C N^{-\frac{1}{2}}\left\|\partial_\theta\left( \tilde{a}_1(\theta^{\frac{1}{\lambda}}) \int_0^{\theta^{\frac{1}{\lambda}}} e^*(\eta) d \eta\right)\right\|_{\infty} \notag\\
		&\leq C N^{-\frac{1}{2}}\left\|\frac{1}{\lambda} \theta^{\frac{1}{\lambda}-1}\partial_\theta\tilde{a}_1(\theta^{\frac{1}{\lambda}})\int_0^{\theta^{\frac{1}{\lambda}}} e^*(\eta) d \eta    \right\|_{\infty} 
		+ C N^{-\frac{1}{2}}\left\|\frac{1}{\lambda} \theta^{\frac{1}{\lambda}-1} e^*\left(\theta^{\frac{1}{\lambda}}\right)+\int_0^{\theta^{\frac{1}{\lambda}}} e^*(\eta) d \eta\right\|_{\infty}\notag\\
		& \leq C N^{-\frac{1}{2}}\left\|\theta^{\frac{1}{\lambda}}e^*(\theta)\right\|_{\infty}+C N^{-\frac{1}{2}}\left\|\frac{1}{\lambda} \theta^{\frac{1}{\lambda}-1} e^*\left(\theta^{\frac{1}{\lambda}}\right)+\int_0^{\theta^{\frac{1}{\lambda}}} e^*(\eta) d \eta\right\|_{\infty} \notag\\
		&\leq C N^{-\frac{1}{2}}\left(\left\|e^*\left(\theta^{\frac{1}{\lambda}}\right)\right\|_{\infty}+
		\left\|e^*(\theta)\right\|_{\infty}\right)
		\leq C N^{-\frac{1}{2}}\left\|e^*(\theta)\right\|_{\infty}.\label{eq_E4_infty}
	\end{align}
Similarly, estimate $\left\|E_5(\theta)\right\|_{\infty},\left\|E_9(\theta)\right\|_{\infty}$ respectively 
	\begin{align}
		&\left\|E_5(\theta)\right\|_{\infty}\leq C N^{-\frac{1}{2}}\left\|e^*(\theta)\right\|_{\infty},\label{eq_E5_infty}\\
		&\left\|E_9(\theta)\right\|_{\infty}\leq C N^{-\frac{1}{2}}\left\|e^*(\theta)\right\|_{\infty}.\label{eq_E9_infty}
	\end{align}
When 
$\bar{K}_2(\theta, \varepsilon\eta) \in C^m(I \times I)$, we can deduce from  \eqref{eq_Interpolation_operator_sum}, \eqref{eq_Crk}, and Lemmas \ref{lemma_J6} and \ref{lemma_Lesbegue_constant} that
		\begin{align}
		\left\|E_7(\theta)\right\|_{\infty} 
		& =\max _{\theta \in I}\left|I_{N, \lambda}^{\alpha, \beta}(\mathcal{K}_2 e)(\theta)-(\mathcal{K}_2 e)(\theta)\right|
		=\max _{z^{1 / \lambda}=\theta \in I}\left|I_{N, 1}^{\alpha, \beta}(\mathcal{K}_2 e)\left(z^{1 / \lambda}\right)-(\mathcal{K}_2 e)\left(z^{1 / \lambda}\right)\right| \notag\\ 
		& =\left\|\left(I_{N, 1}^{\alpha, \beta}-I\right)(\mathcal{K}_2 e)\left(z^{1 / \lambda}\right)\right\|_{\infty}\notag\\
		& =\left\|\left(I_{N, 1}^{\alpha, \beta}-I\right)\left[(\mathcal{K}_2 e)\left(z^{1 / \lambda}\right)-\mathcal{T}_N(\mathcal{K}_2 e)\left(z^{1 / \lambda}\right)\right]\right\|_{\infty}\notag \\ 
		& \leq\left(\left\|I_{N, 1}^{\alpha, \beta}\right\|_{\infty}+1\right)\left\|(\mathcal{K}_2 e)\left(z^{1 / \lambda}\right)-\mathcal{T}_N(\mathcal{K}_2 e)\left(z^{1 / \lambda}\right)\right\|_{\infty} \notag\\ 
		& \leq c N^{-\kappa} \log N\|e(\theta)\|_{\infty}, \notag\\
		&\leq C N^{-\kappa} \log N\left(\|{e}^{*}(\theta)\|_{\infty}+\left\|E_8(\theta)\right\|_{\infty}+\left\|E_9(\theta)\right\|_{\infty}\right),\quad 0<\kappa<1-\mu.
		\end{align}
It can be obtained from the similar way that
		\begin{align}
		\left\|E_{6}(\theta)\right\|_{\infty}=
		& \leq c N^{-\kappa} \log N\|e(\theta)\|_{\infty}\notag\\
		&\leq C N^{-\kappa} \log N\left(\|{e}^{*}(\theta)\|_{\infty}+\left\|E_8(\theta)\right\|_{\infty}+\left\|E_9(\theta)\right\|_{\infty}\right).\label{eq_E7_infty}
	\end{align}
In conclusion, from \eqref{eq_E1_infty} to \eqref{eq_E7_infty} together with \eqref{theorem_normcontrol_e*_infty} and \eqref{theorem_normcontrol_e_infty}, we can obtain the estimates for $\left\|e^*(\theta)\right\|_{\infty}$ and $\left\|e(\theta)\right\|_{\infty}$ as follows
	\begin{align}
		\left\|e^*(\theta)\right\|_{\infty}\leq &CN^{-m}\Bigl\{\log N\bigl(\mathcal{K}^*\left\|\varphi(\theta)\right\|_{\infty}
		+N^{\frac{1}{2}-\kappa}\left\|\partial_\theta^m \varphi\left(\theta^{\frac{1}{\lambda}}\right)\right\|_{0, \omega^{\alpha+m,\beta+m,1}}\bigr)+N^{\frac{1}{2}}\Phi\Bigr\},\\
		\left\|e(\theta)\right\|_{\infty}\leq&CN^{-m}\Bigl\{\log 	N\bigl(\mathcal{K}^*\left\|\varphi(\theta)\right\|_{\infty}
		+N^{\frac{1}{2}-\kappa}\left\|\partial_\theta^m 	\varphi\left(\theta^{\frac{1}{\lambda}}\right)\right\|_{0, \omega^{\alpha+m,\beta+m,1}}\bigr)
		+N^{\frac{1}{2}}\Phi\Bigr\}.
	\end{align}
\end{proof}
Thereupon, the error estimate under the $L_{\omega^{\alpha,\beta,\lambda}}^{2}$-norm will be derived.

\subsection{Error estimate in $L_{\omega^{\alpha,\beta,\lambda}}^{2}$-norm}
\begin{theorem}\label{theorem_L2_norm}
If the hypotheses keep the same as Theorem \ref{theorem_infty_norm}, then the estimates of  $L_{\omega^{\alpha,\beta,\lambda}}^{2}$-norm will become:
	\begin{align}
	\left\|e^*(\theta)\right\|_{0,\omega^{\alpha,\beta,1}} \leq&
	CN^{-m}\Bigl\{\left(N^{-\kappa}\log N+1\right)\mathcal{K}^*\left\|\varphi(\theta)\right\|_{\infty} \notag\\
	&+N^{\frac{1}{2}-\kappa}\left\|\partial_\theta^m \varphi\left(\theta^{\frac{1}{\lambda}}\right)\right\|_{0, \omega^{\alpha+m,\beta+m,1}}
	+\left(N^{\frac{1}{2}-\kappa}+1\right)\Phi\Bigr\},\label{theorem_e*_L2}\\
	\left\|e(\theta)\right\|_{0,\omega^{\alpha,\beta,1}} \leq&
	CN^{-m}\Bigl\{\left(N^{-\kappa}\log N+1\right)\mathcal{K}^*\left\|\varphi(\theta)\right\|_{\infty} \notag\\
	&+\left(N^{\frac{1}{2}-\kappa}+1\right)\left\|\partial_\theta^m \varphi\left(\theta^{\frac{1}{\lambda}}\right)\right\|_{0, \omega^{\alpha+m,\beta+m,1}}
	+\left(N^{\frac{1}{2}-\kappa}+1\right)\Phi\Bigr\}. \label{theorem_e_L2}
	\end{align}
\end{theorem}
\begin{proof}
From Lemmas \ref{lemma_Interpolation_L^2} and \ref{lemma_E6E7_L^2}, we can demonstrate the following error analysis for the weighted $L^2$-norm, similar to Theorem \ref{theorem_infty_norm}:
	\begin{align}
		\left\|E_1(\theta)\right\|_{0, \omega^{\alpha, \beta, \lambda}}&=\left\|\sum_{i=0}^N I_{i, 1} F_{i,\lambda}(\theta)\right\|_{0, \omega^{\alpha, \beta,\lambda}} \leq C\max _{0 \leq i \leq N}\left|I_{i, 1}\right|\notag \\
		& \leq C N^{-m} \max\limits_{0\leq i\leq N}\left\|\partial_\theta^m\widetilde{K}_1\left(\theta_i,  \eta_i(\cdot)\right)\right\|_ {0,\omega^{m-\mu,{m+\frac{1}{\lambda}}-1,1}}\left(\|e(\theta)\|_{\infty}+\|\varphi(\theta)\|_{\infty}\right) \notag\\
		&\leq C N^{-m}\max\limits_{0\leq i\leq N}\left\|\partial_\theta^m \widetilde{K}_1\left(\theta_i,  \eta_i(\cdot)\right)\right\|_ {0,\omega^{m-\mu,{m+\frac{1}{\lambda}}-1,1}}\notag\\
		&\quad\left(\|{e}^{*}(\theta)\|_{\infty}+\left\|E_8(\theta)\right\|_{\infty}+\left\|E_9(\theta)\right\|_{\infty}+\|\varphi(\theta)\|_{\infty}\right),\label{eq_E1_L2}\\
		\left\|E_2(\theta)\right\|_{0, \omega^{\alpha, \beta,\lambda}} &\leq C N^{-m}\max\limits_{0\leq i\leq N}\left\|\partial_\theta^m \widetilde{K}_2\left(\theta_i, \varepsilon \eta_i(\cdot)\right)\right\|_ {0,\omega^{m-\mu,{m+\frac{1}{\lambda}}-1,1}}\left(\|e(\theta)\|_{\infty}+\|\varphi(\theta)\|_{\infty}\right)\notag\\
		&\leq C N^{-m}\max\limits_{0\leq i\leq N}\left\|\partial_\theta^m \widetilde{K}_2\left(\theta_i, \varepsilon \eta_i(\cdot)\right)\right\|_ {0,\omega^{m-\mu,{m+\frac{1}{\lambda}}-1,1}}\notag\\
		&\quad\left(\|{e}^{*}(\theta)\|_{\infty}+\left\|E_8(\theta)\right\|_{\infty}+\left\|E_9(\theta)\right\|_{\infty}+\|\varphi(\theta)\|_{\infty}\right),\label{eq_E2_L2}
	\end{align}
	\begin{align}
		\left\|E_3(\theta)\right\|_{0, \omega^{\alpha, \beta, \lambda}} &\leq C N^{-m}\left\|\partial_\theta^{m+1} \varphi\left(\theta^{\frac{1}{\lambda}}\right)\right\|_{0, \omega^{\alpha+m,\beta+m,1}},\label{eq_E3_L2}\\
		\left\|E_8(\theta)\right\|_{0, \omega^{\alpha, \beta,\lambda}} &\leq C N^{-m} \left\|\partial_\theta^{m} \varphi\left(\theta^{\frac{1}{\lambda}}\right)\right\|_{0, \omega^{\alpha+m,\beta+m,1}},\label{eq_E8_L2}\\
		\left\|E_4(\theta)\right\|_{0, \omega^{\alpha, \beta,\lambda}} &\leq C N^{-1} \left\| e^*\left(\theta^{\frac{1}{\lambda}}\right)\right\|_{0, \omega^{\alpha, \beta,\lambda}} \leq C N^{-1}\left\| e^*(\theta)\right\|_{\infty},\label{eq_E4_L2}
	\end{align}
	\begin{align}
		\left\|E_5(\theta)\right\|_{0, w^{\alpha, \beta, \lambda}} &\leq C N^{-1}\left\|e^*(\theta)\right\|_{\infty},\label{eq_E5_L2}\\
		\left\|E_9(\theta)\right\|_{0, \omega^{\alpha, \beta,\lambda}}&\leq C N^{-1}\left\|e^*(\theta)\right\|_{\infty}\label{eq_E9_L2},\\
	\end{align}
		\begin{align}
			&\left\|\left(I_{N, \lambda}^{\alpha, \beta}-I\right)(\mathcal{K}_ie)(\theta)\right\|_{0,\omega^{\alpha, \beta, \lambda}} \notag\\
			& =\left\{\int_0^1\left[\left(I_{N, \lambda}^{\alpha, \beta}-I\right)(\mathcal{K}_ie)(\theta)\right]^2 \lambda\left(1-\theta^\lambda\right)^\alpha \theta^{(\beta+1) \lambda-1} d \theta\right\}^{1 / 2} \notag\\
			& =\left\{\int_0^1\left[\sum_{i=0}^N(\mathcal{K}_ie)\left(\theta_i\right)F_{j,\lambda}(\theta)-(\mathcal{K}_ie)(\theta)\right]^2\left(1-\theta^\lambda\right)^\alpha \theta^{\beta \lambda} d\theta^\lambda\right\}^{1 / 2}\notag \\
			& =\left\{\int_0^1\left[\sum_{i=0}^N(\mathcal{K}_ie)\left(z_i^{\frac{1}{\lambda}}\right)F_{j,1}(z)-(\mathcal{K}_ie)\left(z^{\frac{1}{\lambda}}\right)\right]^2(1-z)^\alpha z^\beta d z\right\}^{1 / 2}\notag \\
			& =\left\|\left(I_{N, 1}^{\alpha, \beta}-I\right)(\mathcal{K}_ie)\left(z^{\frac{1}{\lambda}}\right)\right\|_{0,\omega^{\alpha, \beta, 1}}
		\end{align}
	\begin{align}
		\left\|E_{6}(\theta)\right\|_{0, \omega^{\alpha, \beta,\lambda}}&=\left\|\left(I_{\beta, \lambda}^{\alpha, \beta}-I\right)\left(\mathcal{K}_1 e\right)(\theta)\right\|_{0, \omega^{\alpha, \beta, \lambda}} \notag\\
		& =\left\|\left(I_{N, 1}^{\alpha, \beta}-I\right)\left(\mathcal{K}_1 e\right)\left(z^{\frac{1}{\lambda}}\right) \right\|_{0, \omega^{\alpha, \beta, 1}}\notag\\
		& =\left\|\left(I_{N, 1}^{\alpha, \beta}-I\right)\left(\mathcal{K}_1 e-\mathcal{T}_N \mathcal{K}_1 e\right)\left(z^{\frac{1}{\lambda}}\right)\right\|_{0, w^{\alpha, \beta,1}} \notag\\
		& \leq \left\|\left(I_{N, 1}^{\alpha, \beta}\left(\mathcal{K}_1e-\mathcal{T}_N \mathcal{K}_1 e\right)\left(z^{\frac{1}{\lambda}}\right) \right\|_{0, \omega^{\alpha, \beta,1}}\right. +\left\|\left(\mathcal{K}_1 e-\mathcal{T}_N \mathcal{K}_1 e\right)\left(z^{\frac{1}{\lambda}}\right)\right\|_{0, \omega^{\alpha, \beta,1}}\notag\\
		& \leq C \left\|\left(\mathcal{K}_1 e-\mathcal{T}_N \mathcal{K}_1 e\right)\left(z^{\frac{1}{\lambda}}\right) \right\|_{\infty}\notag\\
		& \leq C N^{-\kappa} \left\| \mathcal{K}_1 e\left(z^{\frac{1}{\lambda}}\right) \right\|_{0, \kappa} \leq C N^{-\kappa}\left\| e(\theta) \right\|_{\infty},\notag\\
		& \leq CN^{-\kappa}\left(\|{e}^{*}(\theta)\|_{\infty}+\left\|E_8(\theta)\right\|_{\infty}+\left\|E_9(\theta)\right\|_{\infty}\right), \quad 0\leq\kappa\leq1-\mu,\label{eq_E6_L2}
	\end{align}
	\begin{align}
		\left\|E_{7}(\theta)\right\|_{0, \omega^{\alpha, \beta,\lambda}}& \leq C N^{-\kappa}\left(\|{e}^{*}(\theta)\|_{\infty}+\left\|E_8(\theta)\right\|_{\infty}+\left\|E_9(\theta)\right\|_{\infty}\right), \quad 0\leq\kappa\leq1-\mu.\label{eq_E7_L2}
	\end{align}
It follows from \eqref{eq_E1_L2}-\eqref{eq_E7_L2},  together with \eqref{theorem_normcontrol_e*_w}, \eqref{theorem_normcontrol_e_w}, \eqref{eq_e_e*} and \eqref{eq_E8_infty} that the error estimates for $\left\|e^*(\theta)\right\|_{0,\omega^{\alpha,\beta,1}}$ and $\left\|e(\theta)\right\|_{0,\omega^{\alpha,\beta,1}}$ are:
	\begin{align}
		\left\|e^*(\theta)\right\|_{0,\omega^{\alpha,\beta,1}} \leq&
		CN^{-m}\Bigl\{\left(N^{-\kappa}\log N+1\right)\mathcal{K}^*\left\|\varphi(\theta)\right\|_{\infty} \notag\\
		&+N^{\frac{1}{2}-\kappa}\left\|\partial_\theta^m \varphi\left(\theta^{\frac{1}{\lambda}}\right)\right\|_{0, \omega^{\alpha+m,\beta+m,1}}
		+\left(N^{\frac{1}{2}-\kappa}+1\right)\Phi\Bigr\},\\
		\left\|e(\theta)\right\|_{0,\omega^{\alpha,\beta,1}} \leq&
		CN^{-m}\Bigl\{\left(N^{-\kappa}\log N+1\right)\mathcal{K}^*\left\|\varphi(\theta)\right\|_{\infty} \notag\\
		&+\left(N^{\frac{1}{2}-\kappa}+1\right)\left\|\partial_\theta^m \varphi\left(\theta^{\frac{1}{\lambda}}\right)\right\|_{0, \omega^{\alpha+m,\beta+m,1}}
		+\left(N^{\frac{1}{2}-\kappa}+1\right)\Phi\Bigr\}.
	\end{align}
\end{proof}
Finally, we have already proved the error estimates for $L^{\infty}$-norm and the weighted $L^2$-norm.

\begin{remark}\label{remark_form of solution}
	It is difficult to seek the form of the exact solutions for \eqref{eq_prime}, though Brunner\cite{Brunner}  has already derived the form of the exact solutions for the second kind VIDEs and VIEs with proportional delays. 
	Consider a special situation that $b_1(t)=0$, $K_2(t,\tau)=0$, we can apply Theorem 7.1.4 in \cite{Brunner} that the exact solution can be represented as 
	\begin{equation}
		\qquad \qquad \qquad \qquad \qquad \qquad y(t)=\sum_{(i,j)_{\mu}} \Psi_{i, j}(\mu) t^{i+j(2-\mu)},
	\end{equation}
where the notation $(i,j)_{\mu}:=\lbrace (i,j):i,j \in \mathbb{N}_0, i+j(2-\mu) < m+1, j, k \text{ are non-negative integers} \rbrace$. Moreover, the coefficients $\gamma_{k, j}(\mu)$ are related to $\mu$ in Theorem 7.1.4 in \cite{Brunner}. Additionally, the theorem also assume that $a_1, b_1, f_1 \in C^m(I_0) $ and $K_1, K_2 \in C^m(D)(m \geq 1)$. In next section, we will not only consider the special situation.
	
Despite this, the exact solutions still exhibit limited regularity and $y^{\prime\prime}(t)$ is unbounded at $t=0^{+}$, that is $\left|y^{\prime \prime}(t)\right| \leq C t^{-\mu}$. The existence of delay terms cannot improve the regularity of the solution\cite{WVIDEs_hp_QinYu}\cite{Remark_2}. Furthermore, if the exact solution is unknown, it can be made by the following strategy: choose $\lambda=\frac{1}{q}$ of $\mu=\frac{p}{q}$ with integers $p$ and $q$.
\end{remark}

\section{Numerical experiments}\label{section_Numerical experiments}
In this section, some numerical experiments are showed in order to verify the proposed numerical method. All the errors will be presented in the $L^{\infty}(I)$ and $L^{2}_{\omega^{\alpha,\beta,\lambda}}(I)$-norm in the following text, where $\alpha$ and $\beta$ are related coefficients. 
The main contributions of these numerical results is that when the functions $y\left(t^{\frac{1}{\lambda}}\right)$ and $y^{\prime}\left(t^{\frac{1}{\lambda}}\right)$ are smooth enough, the proposed method achieves exponential convergence. All left figures in the following text represent $\left\|e(\theta)\right\|_{0,\omega^{\alpha,\beta,1}}$ and $\|e(\theta)\|_{\infty}$, all right figures represent the trend for $\left\|e^*(\theta)\right\|_{0,\omega^{\alpha,\beta,1}}$ and $\|e^*(\theta)\|_{\infty}$ with the change of $N$. All upper tables represent $\left\|e(\theta)\right\|_{0,\omega^{\alpha,\beta,1}}$ and $\|e(\theta)\|_{\infty}$, all lower tables represent $\left\|e^*(\theta)\right\|_{0,\omega^{\alpha,\beta,1}}$ and $\|e^*(\theta)\|_{\infty}$..
\begin{example}\label{example_second_NO1}
Consider the following problem: 
	\begin{equation}
		\left\{\begin{array}{l}
			y^{\prime}(t)=-y(t)+y(\varepsilon t)+f_1(t)
			-\int_0^t (t-s)^{-\mu}   e^{s^{1-\mu}} y(s) d s+\int_0^{\varepsilon t} (\varepsilon t-\tau)^{-\mu} e^{\tau^{1-\mu}}  y(\tau) d \tau, \quad t \in[0,1],  \\
			y(0)=0 .
		\end{array}\right.
	\end{equation}
 where  the exact solution is $y(t)=t e^{-t^{1-\mu}}$ with $f_1(t)=\left(1-(1-\mu) t^{1-\mu}+t\right) e^{-t^{1-\mu}}+(1+e^{2-\mu})B(1-\mu, 2)t^{2-\mu}-(\varepsilon t) e^{-(\varepsilon t)^{1-\mu}}, B(\cdot, \cdot)$ for the Beta function.
 
In this example, we take $\lambda=\mu=0.5 , T=1$. Through transformation $t=\theta^{\frac{1}{\lambda}}$, $y\left(\theta^{\frac{1}{\lambda}}\right)$ and $y^{\prime}\left(\theta^{\frac{1}{\lambda}}\right)$ are both analytical even if $y(t)$ and $y^{\prime}(t)$ are weakly
 singular at $t = 0$. 
 
 In Figures \ref{Fig_second_NO1_0.5} and \ref{Fig_second_NO1_1}, we can see the errors and convergence results intuitively. Besides, we the specific datas of the error bounds through Tables \ref{tabular_second_NO1_e} and \ref{tabular_second_NO1_e*}.
When $\lambda=1$, the namely polynomial collocation method can only arrive at algebraic convergence, with a minimum error of $10^{-6}$ magnitude at $N=50$. When $\lambda=\frac{1}{2}$, exponential convergence can be easily achieved, with a minimum error of $7.07828\times10^{-11}$ at $N=12$. It is evident from comparison that this method is much more accurate.

		\begin{figure} 
		
		\subfloat[$\lambda=\frac{1}{2},\left\|e(\theta)\right\|_{0,\omega^{\alpha,\beta,1}},\|e(\theta)\|_{\infty}$]{\includegraphics[width=0.5\textwidth]{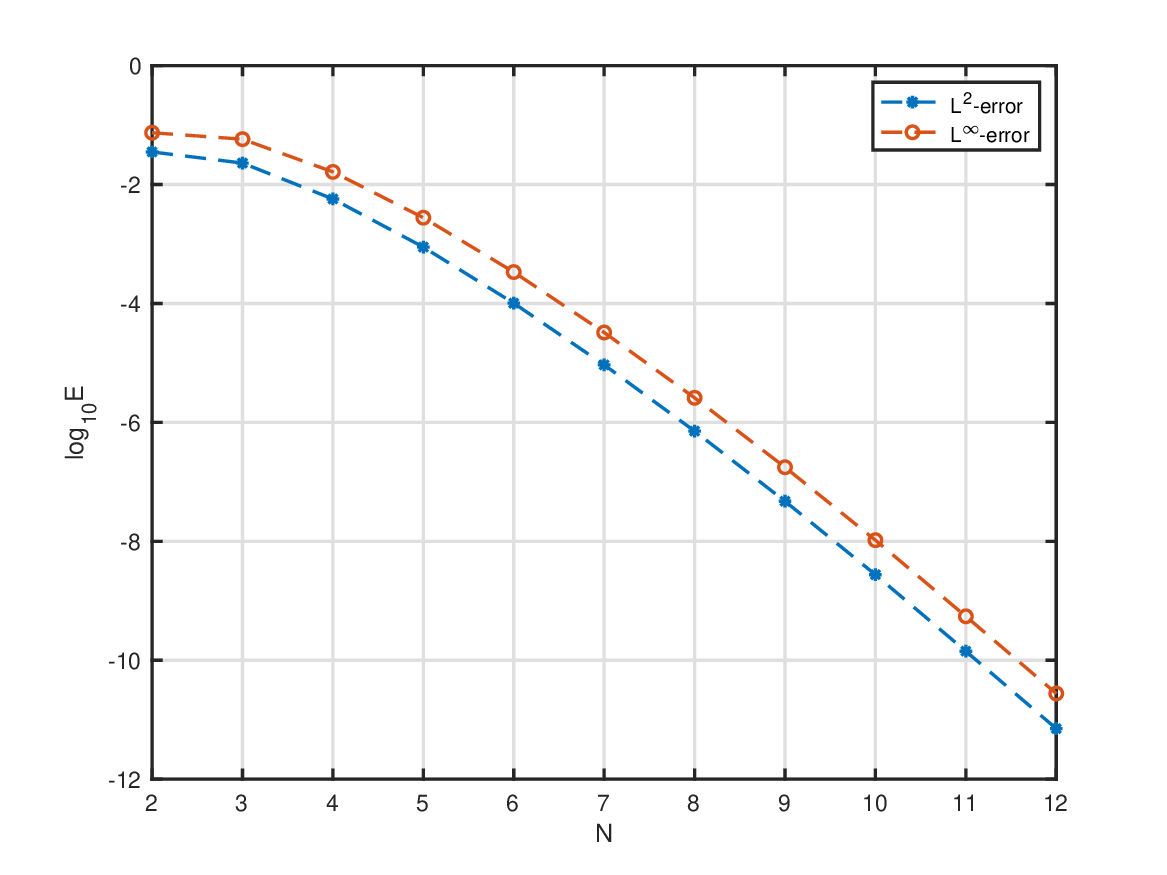}}%
		\hfill
		\subfloat[$\lambda=\frac{1}{2},\left\|e^*(\theta)\right\|_{0,\omega^{\alpha,\beta,1}},\|e^*(\theta)\|_{\infty}$]{\includegraphics[width=0.5\textwidth]{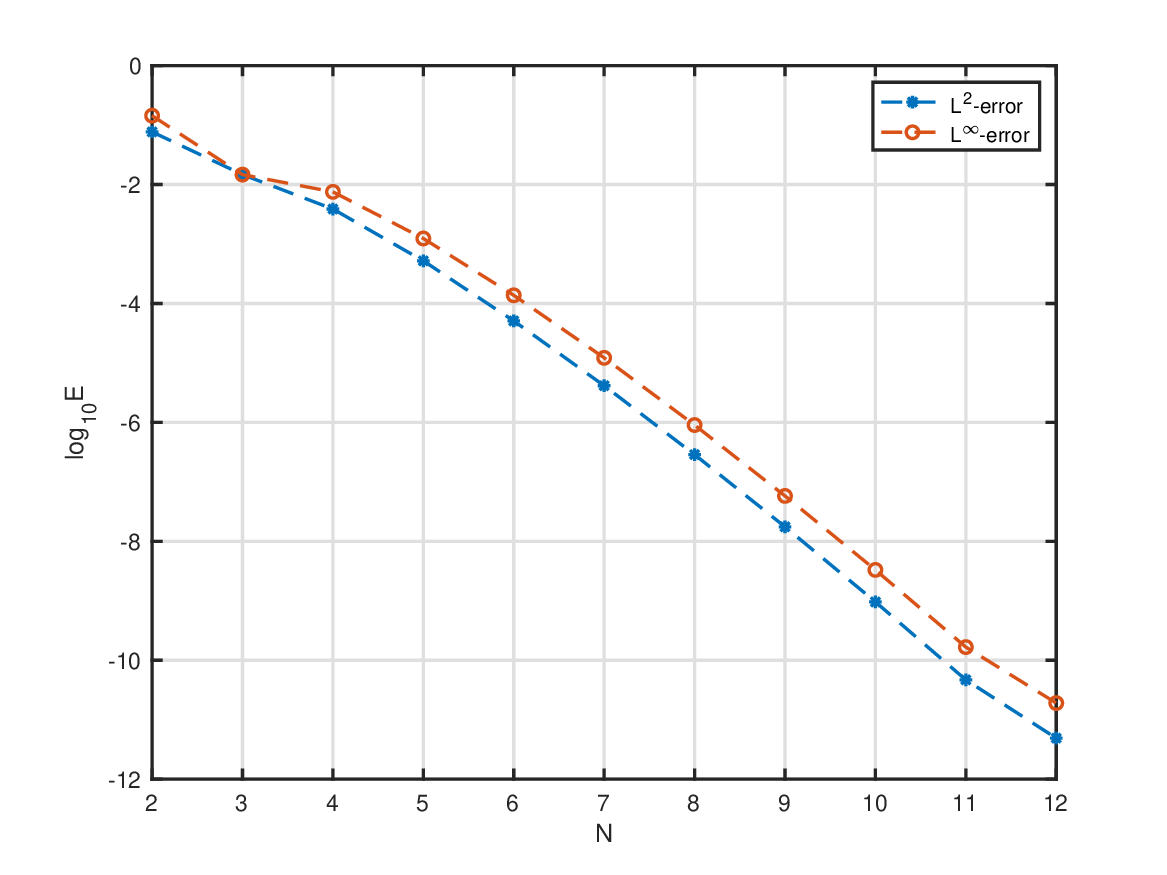}}
		\caption{\bf{Example \ref{example_second_NO1}} with $\lambda=\frac{1}{2}$}
		\label{Fig_second_NO1_0.5}
	\end{figure}
	
	\begin{figure}
		\subfloat[$\lambda=1,\left\|e(\theta)\right\|_{0,\omega^{\alpha,\beta,1}},\|e(\theta)\|_{\infty}$]{\includegraphics[width=0.5\textwidth]{ 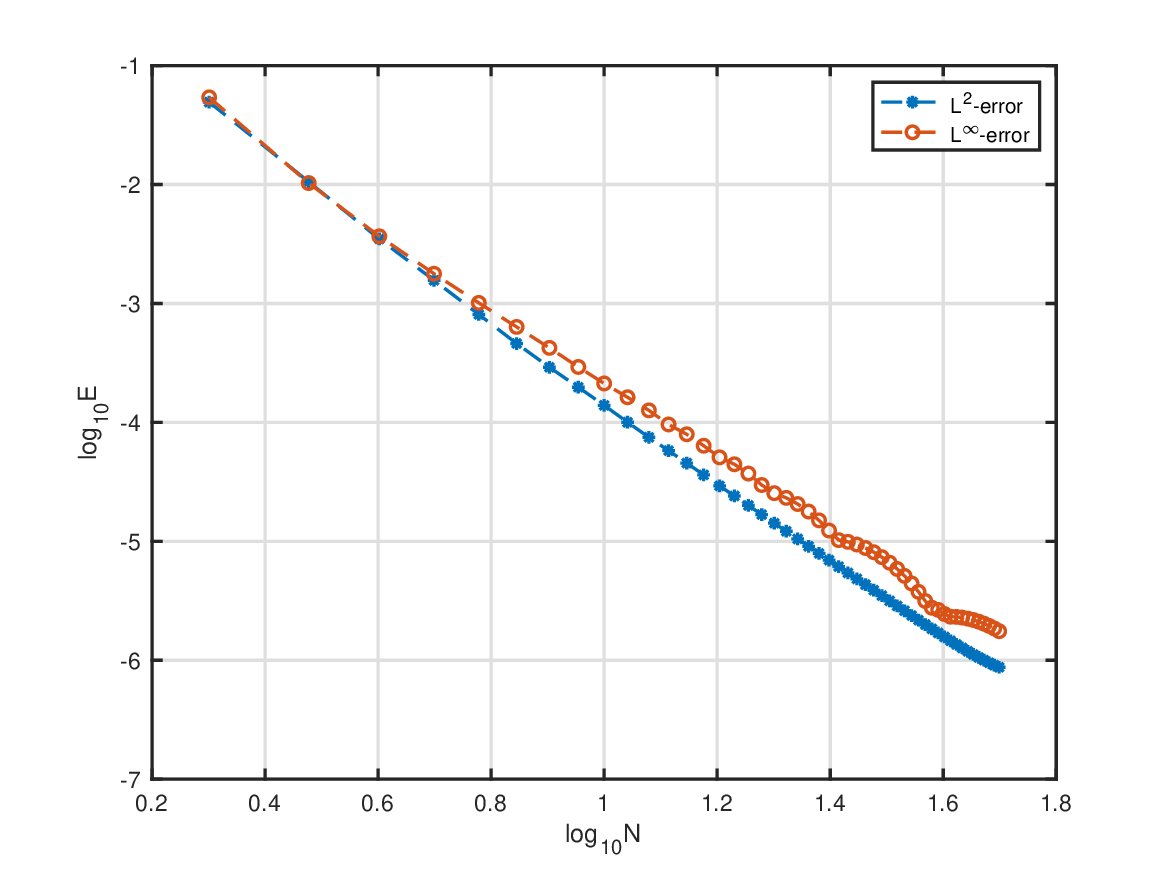}}%
		\hfill
		\subfloat[$\lambda=1,\left\|e^*(\theta)\right\|_{0,\omega^{\alpha,\beta,1}},\|e^*(\theta)\|_{\infty}$]{\includegraphics[width=0.5\textwidth]{ 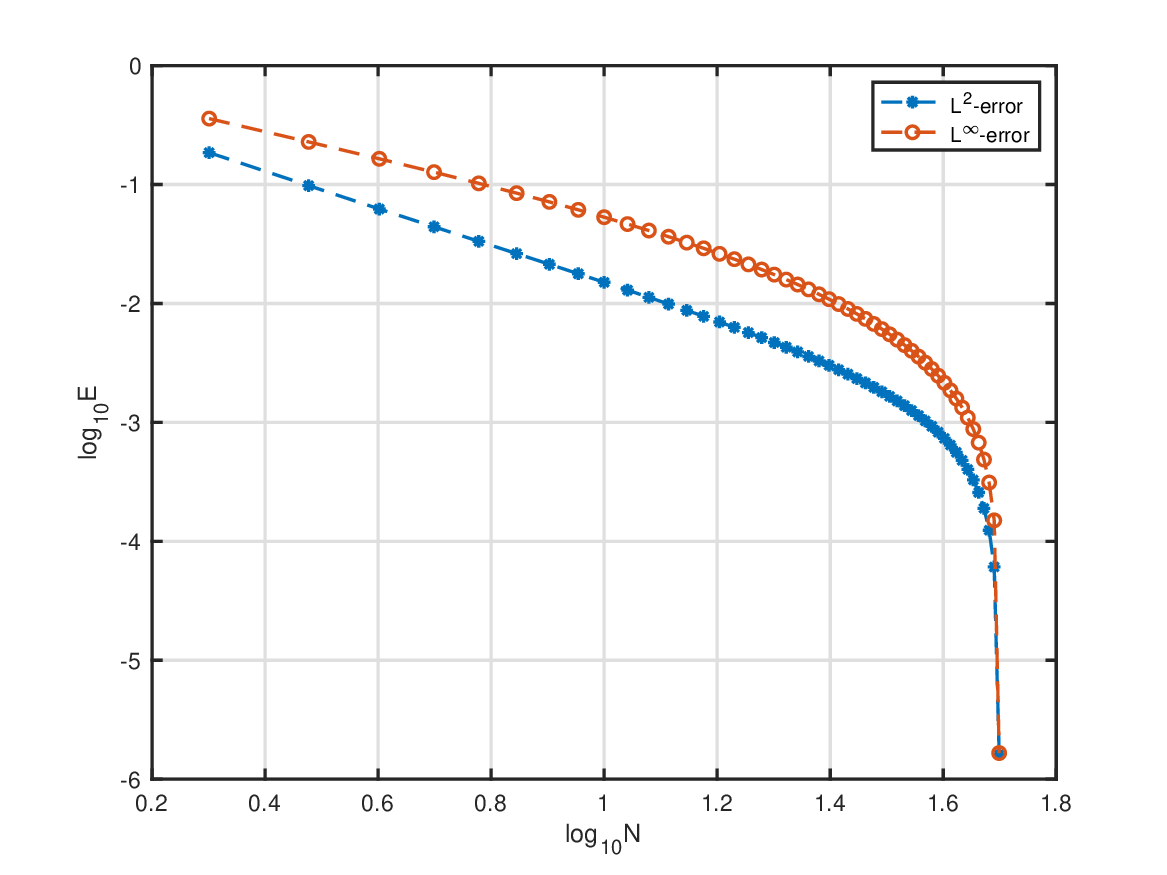}}
		\caption{\bf{Example \ref{example_second_NO1}} with $\lambda=1$}
		\label{Fig_second_NO1_1}
	\end{figure}
	
	\begin{table}[!ht]
		\centering
		\caption{\bf{Example \ref{example_second_NO1}} with $\lambda=\frac{1}{2}$: $\left\|e(\theta)\right\|_{0,\omega^{\alpha,\beta,1}}$ and $\|e(\theta)\|_{\infty}$.}
		\begin{tabular}{llllll}
			\hline$N$& 4 & 6 & 8 & 10 & 12\\
			\hline$L^2$-error  & $5.72413 \mathrm{e}-03$ & $1.01173 \mathrm{e}-04$ & $7.13447 \mathrm{e}-07$ & $2.74588 \mathrm{e}-09$ & $7.07828 \mathrm{e}-12$\\
			$L^{\infty}$-error  & $1.63303 \mathrm{e}-02$ & $3.37608\mathrm{e}-04$ & $2.59111 \mathrm{e}-06$ & $1.04489 \mathrm{e}-08$ &$2.76097 \mathrm{e}-11$\\
			\hline
		\end{tabular}
		\label{tabular_second_NO1_e}
	\end{table}
	
	\begin{table}[!ht]
		\centering
		\caption{\bf{Example \ref{example_second_NO1}} with $\lambda=\frac{1}{2}$: $\left\|e^*(\theta)\right\|_{0,\omega^{\alpha,\beta,1}}$ and $\|e^*(\theta)\|_{\infty}$.}
		\begin{tabular}{llllll}
			\hline$N$& 4 & 6 & 8 & 10 & 12\\
			\hline$L^2$-error  & $3.87671 \mathrm{e}-03$ & $5.08647\mathrm{e}-05$ & $2.87371 \mathrm{e}-07$ & $9.58174 \mathrm{e}-10$ & $4.89333 \mathrm{e}-12$\\
			$L^{\infty}$-error& $7.53409 \mathrm{e}-03$ & $1.37196 \mathrm{e}-04$ & $9.03004\mathrm{e}-07$ & $3.30380 \mathrm{e}-09$ & $1.90461 \mathrm{e}-11$\\
			\hline
		\end{tabular}
		\label{tabular_second_NO1_e*}
	\end{table}

\end{example}
\begin{example}\label{example_second_NO2}
Consider the following linear VIDEs:
	\begin{equation}
	\left\{\begin{array}{l}
		y^{\prime}(t)=-y(t)+y(\varepsilon t)+f_1(t)
		-\int_0^t (t-s)^{-\mu}   e^s y(s) d s+\int_0^{\varepsilon t} (\varepsilon t-\tau)^{-\mu} e^{\tau} y(\tau) d \tau, \quad t \in[0,T],  \\
		y(0)=0 .
	\end{array}\right.
	\end{equation}
Let $\varepsilon = 0.6, T=\frac{1}{2}$ and $\mu = \frac{1}{3}$ to test different cases. We set $f_1(t)$ such that the exact solution is $y(t) = t^{2-\mu}e^{-t}$, where $\mu \in (0,1)$, $y(t)$ has a weakly singularity at $t = 0^+$. In this example, $f_1(t)=(2-\mu)t^{1-\mu}e^{-t}+B(1-\mu,3-\mu)t^{3-2\mu}(1+e^{3-2\mu})-(\varepsilon t)^{2-\mu}e^{-\varepsilon t}$.

In fact, considering the structure of the solution, we need to choose $\lambda=\frac{1}{3}$ so that $y\left(t^{\frac{1}{\lambda}}\right)$ and $y^{\prime}\left(t^{\frac{1}{\lambda}}\right)$ meet the requirements. Numerical convegence is shown in Figures \ref{Fig_second_NO2_0.5} and \ref{Fig_second_NO2_1}. Tables \ref{tabular_second_NO2_e} exhibit the errors reaching $10^{-7.5}$. The numerical results are what we expected.
		\begin{figure}
		\subfloat[$\lambda=\frac{1}{3},\left\|e(\theta)\right\|_{0,\omega^{\alpha,\beta,1}},\|e(\theta)\|_{\infty}$]{\includegraphics[width=0.5\textwidth]{ 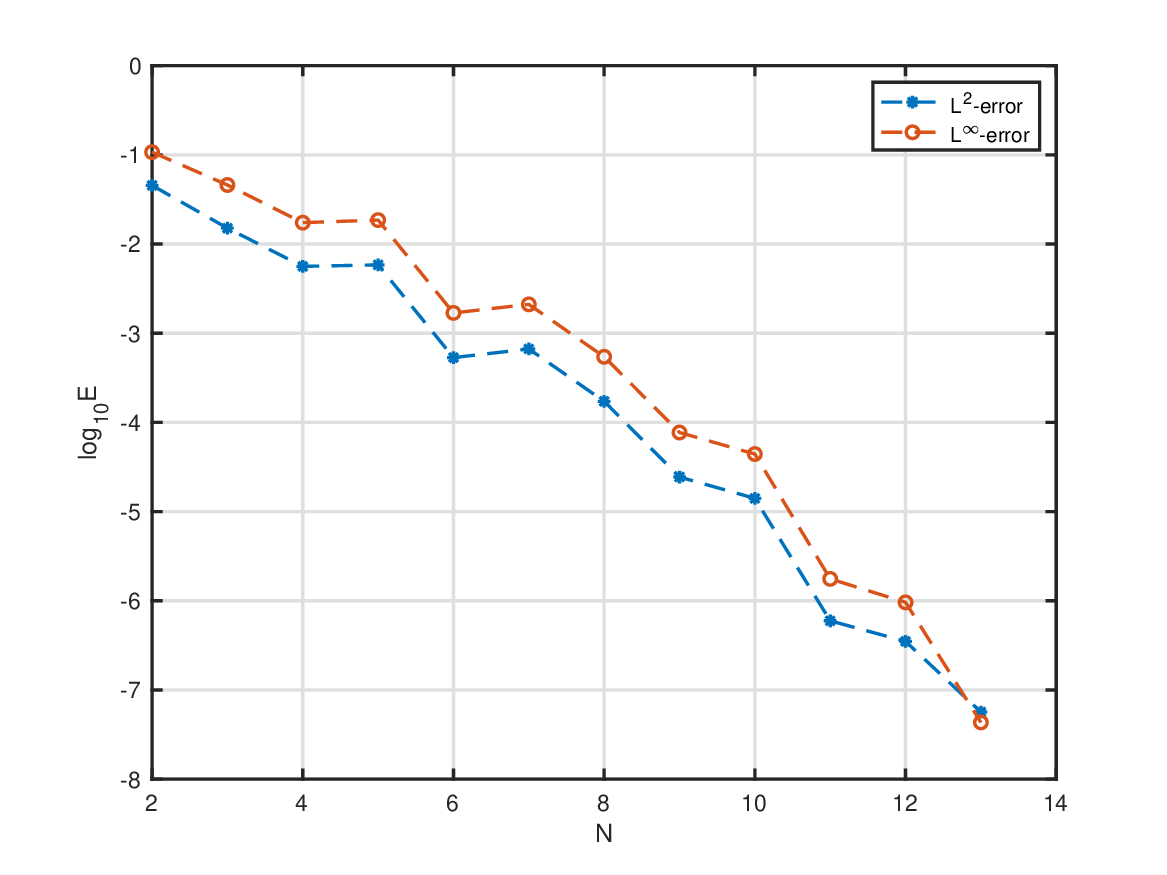}}%
		\hfill
		\subfloat[$\lambda=\frac{1}{3},\left\|e^*(\theta)\right\|_{0,\omega^{\alpha,\beta,1}},\|e^*(\theta)\|_{\infty}$]{\includegraphics[width=0.5\textwidth]{ 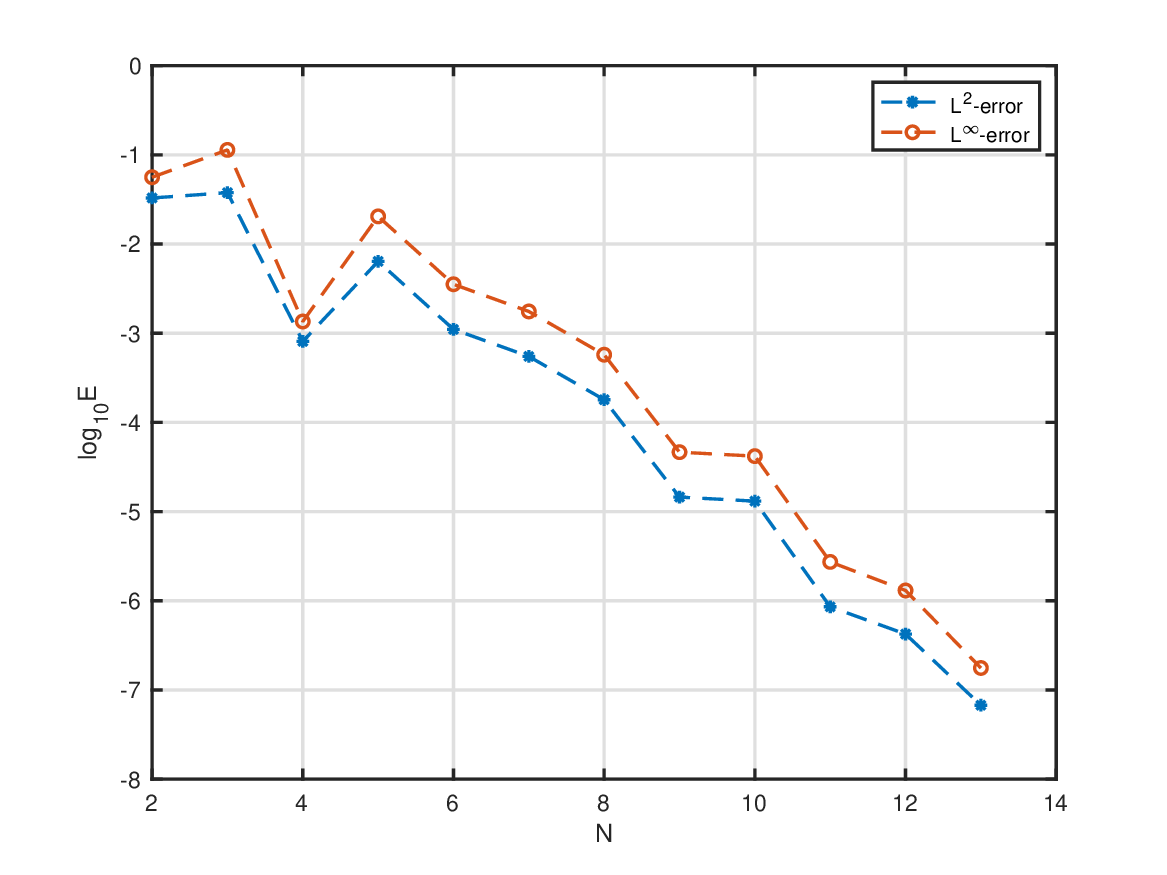}}
		\caption{\bf{Example \ref{example_second_NO2} } with $\lambda=\frac{1}{3}$}
		\label{Fig_second_NO2_0.5}
	\end{figure}
	
	\begin{figure}
		\subfloat[$\lambda=1,\left\|e(\theta)\right\|_{0,\omega^{\alpha,\beta,1}},\|e(\theta)\|_{\infty}$]{\includegraphics[width=0.5\textwidth]{ 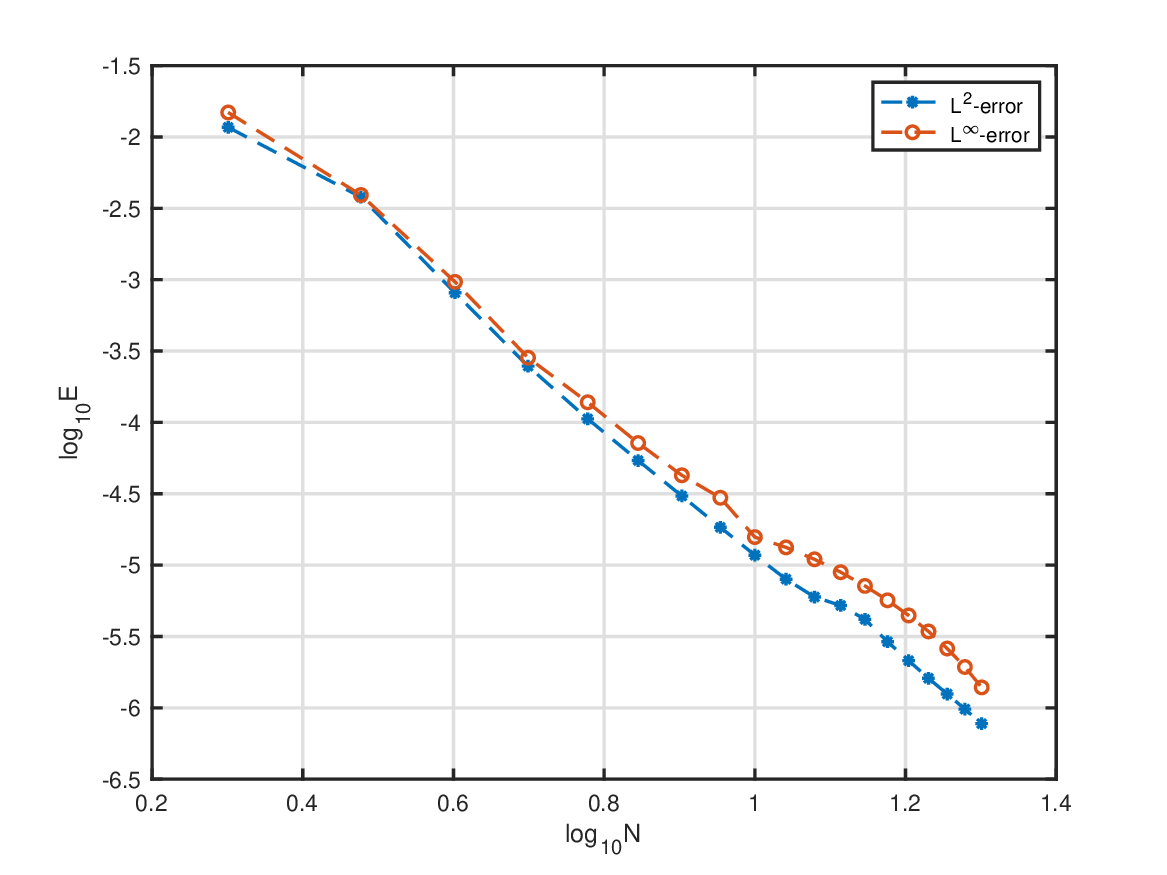}}%
		\hfill
		\subfloat[$\lambda=1,\left\|e^*(\theta)\right\|_{0,\omega^{\alpha,\beta,1}},\|e^*(\theta)\|_{\infty}$]{\includegraphics[width=0.5\textwidth]{ 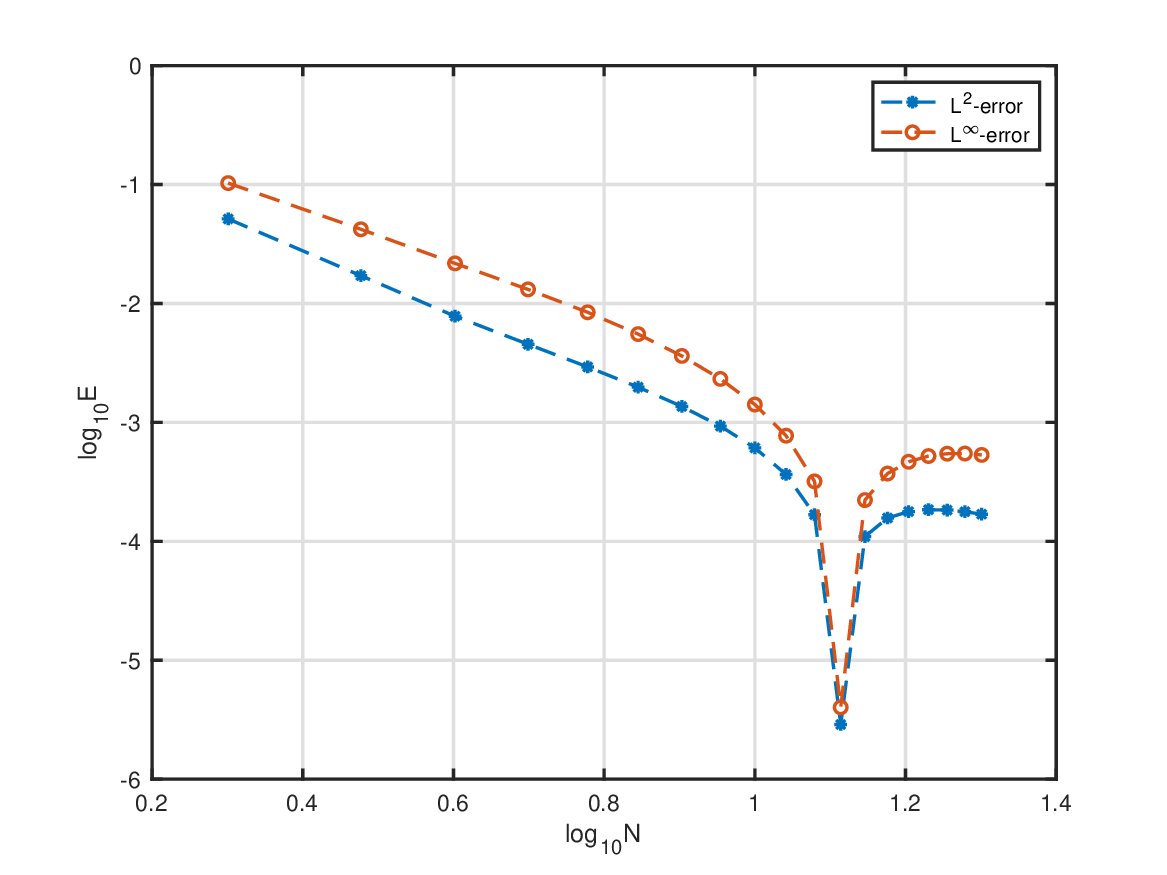}}
		\caption{\bf{Example \ref{example_second_NO2} } with $\lambda=1$}
		\label{Fig_second_NO2_1}
	\end{figure}
	
	\begin{table}[!ht]
		\centering
		\caption{\bf{Example \ref{example_second_NO2} } with $\lambda=\frac{1}{3}$: $\left\|e(\theta)\right\|_{0,\omega^{\alpha,\beta,1}}$ and $\|e(\theta)\|_{\infty}$.}
		\begin{tabular}{llllll}
			\hline$N$ & 5 & 7 & 9 & 11 & 13\\
			\hline$L^2$-error & $5.83579 \mathrm{e}-03$ & $6.65806 \mathrm{e}-04$ & $2.44113 \mathrm{e}-05$ & $5.96438 \mathrm{e}-07$ & $5.66326 \mathrm{e}-08$\\
			$L^{\infty}$-error & $1.85401\mathrm{e}-02$ & $2.10522\mathrm{e}-03$ & $7.70154 \mathrm{e}-05$ &$1.76023 \mathrm{e}-06$ &$4.32917 \mathrm{e}-08$\\
			\hline
		\end{tabular}
		\label{tabular_second_NO2_e}
	\end{table}
	
	\begin{table}[!ht]
		\centering
		\caption{\bf{Example \ref{example_second_NO2}} with  $\lambda=\frac{1}{3}$: $\left\|e^*(\theta)\right\|_{0,\omega^{\alpha,\beta,1}}$ and $\|e^*(\theta)\|_{\infty}$.}
		\begin{tabular}{llllll}
			\hline$N$ & 5 & 7 & 9 & 11 & 13\\
			\hline$L^2$-error & $6.38203\mathrm{e}-03$ & $5.48957 \mathrm{e}-04$ & $1.45378 \mathrm{e}-05$ & $8.56094 \mathrm{e}-07$ & $6.74216 \mathrm{e}-08$\\
			$L^{\infty}$-error & $2.04019\mathrm{e}-02$ & $1.75321\mathrm{e}-03$ & $4.63605 \mathrm{e}-05$ &$2.72662 \mathrm{e}-06$ &$1.76395 \mathrm{e}-07$\\
			\hline
		\end{tabular}
		\label{tabular_second_NO2_e*}
	\end{table}

\end{example}

\begin{example}\label{example_second_NO3}
	Continue to consider the equation in Example \ref{example_second_NO2}, the given function $f_1$ will be
changed into	
	$$
	\begin{aligned}
		f_1(t)=&e^{-t}\left((t^{w_1}(1+w_1-t)+t^{w_2}(1+w_2-t))\right)
		+(t^{1+w_1}+t^{1+w_2})e^{-t}
		-((\varepsilon t)^{1+w_1}+(\varepsilon t)^{1+w_2})e^{-\varepsilon t}\\
		&-B(1-\mu,w_1+2)t^{2-\mu+w_1}(e^{2-\mu+w_1}+1)
		-B(1-\mu,w_2+2)t^{2-\mu+w_2}(e^{2-\mu+w_2}+1).
	\end{aligned}
	$$
	where $T=1$, $\mu=\frac{1}{2}$ and the exact solution $y(t)=(t^{1+w_1}+t^{1+w_2})e^{-t}$. All other parameters remain unchanged.
	
	The purpose is to test the effectiveness of the method for a more complicated situation. It is obvious that we can't guarantee $y\left(t^{\frac{1}{\lambda}}\right)$ and $y^{\prime}\left(t^{\frac{1}{\lambda}}\right)$ are analytic with the selection $w_1=\frac{1}{2}$,$w_2=\sqrt{2}$. Because of Figures \ref{Fig_second_NO3_0.5} and \ref{Fig_second_NO3_1}, we can deduce the numerical results of $\lambda=0.5$ achieve the exponential convergence rates, which are better and much faster than the ones of $\lambda=1$.
	
	 Although this example is more complicated and interval expansion, the convergence results in Figure \ref{Fig_second_NO3_0.5} and Table \ref{tabular_second_NO3_e} achieve better than the ones in Figure \ref{Fig_second_NO2_0.5} and Table \ref{tabular_second_NO2_e}. We guess the reason is that $\lambda=\frac{1}{2}$ is a more suitable parameter for the numerical method.
	
		\begin{figure}
		\subfloat[$\lambda=\frac{1}{2},\left\|e(\theta)\right\|_{0,\omega^{\alpha,\beta,1}},\|e(\theta)\|_{\infty}$]{\includegraphics[width=0.5\textwidth]{ 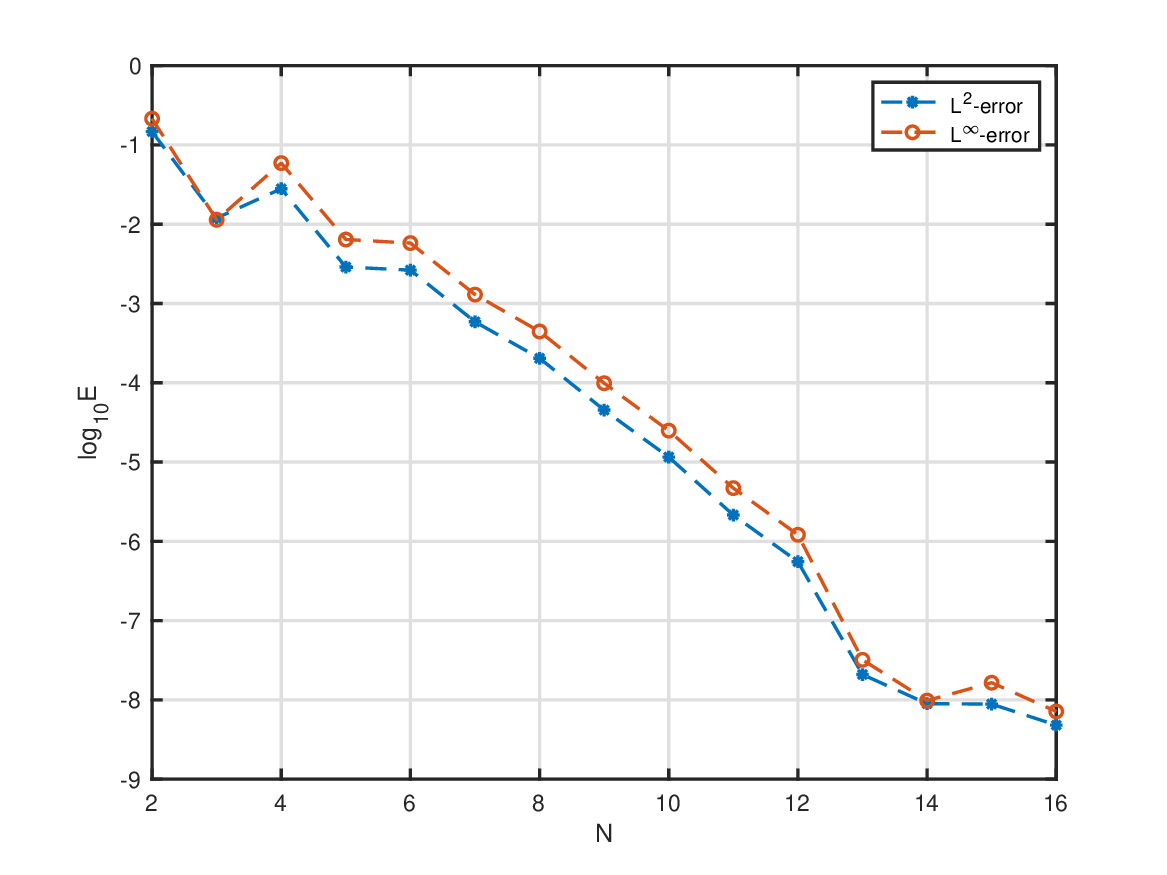}}%
		\hfill
		\subfloat[$\lambda=\frac{1}{2},\left\|e^*(\theta)\right\|_{0,\omega^{\alpha,\beta,1}},\|e^*(\theta)\|_{\infty}$]{\includegraphics[width=0.5\textwidth]{ 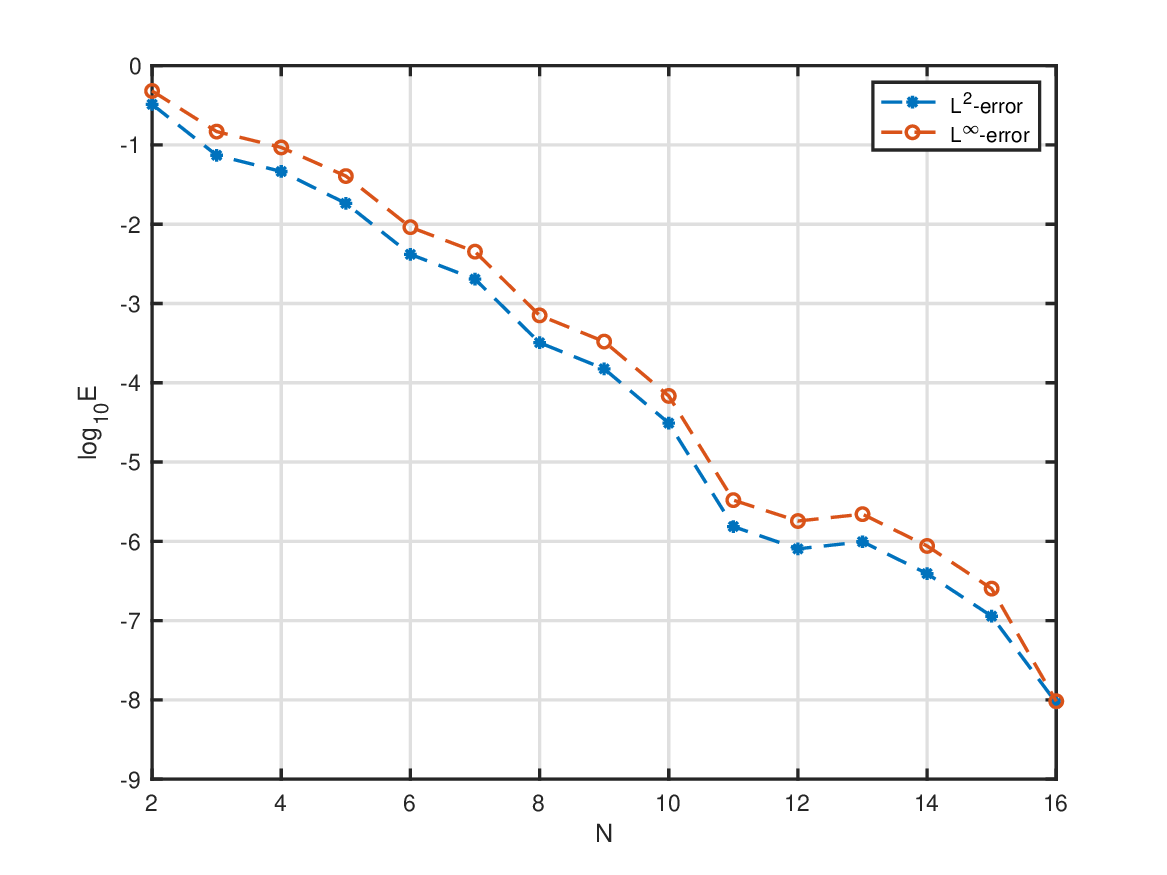}}
		\caption{\bf{Example \ref{example_second_NO3} } with $\lambda=\frac{1}{2}$}
		\label{Fig_second_NO3_0.5}
	\end{figure}
	
	\begin{figure}
		\subfloat[$\lambda=1,\left\|e(\theta)\right\|_{0,\omega^{\alpha,\beta,1}},\|e(\theta)\|_{\infty}$]{\includegraphics[width=0.5\textwidth]{ 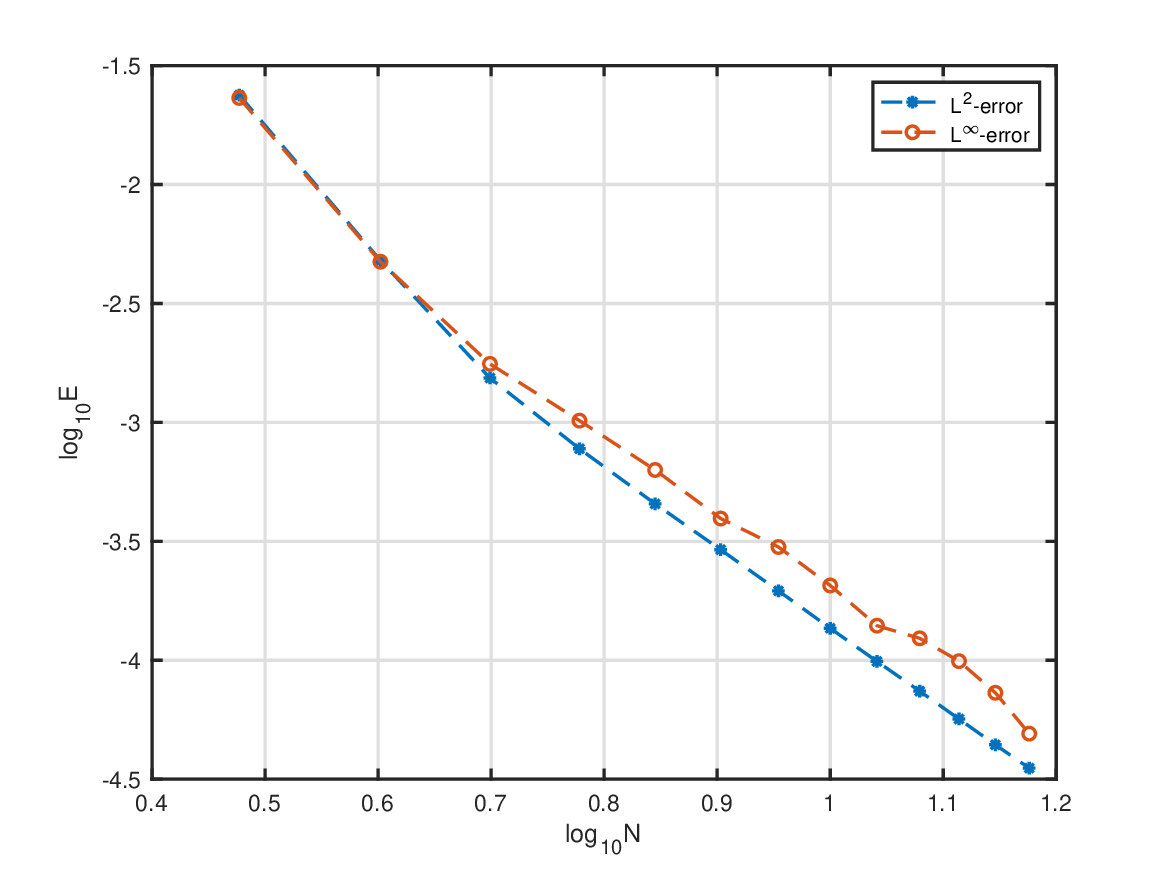}}%
		\hfill
		\subfloat[$\lambda=1,\left\|e^*(\theta)\right\|_{0,\omega^{\alpha,\beta,1}},\|e^*(\theta)\|_{\infty}$]{\includegraphics[width=0.5\textwidth]{ 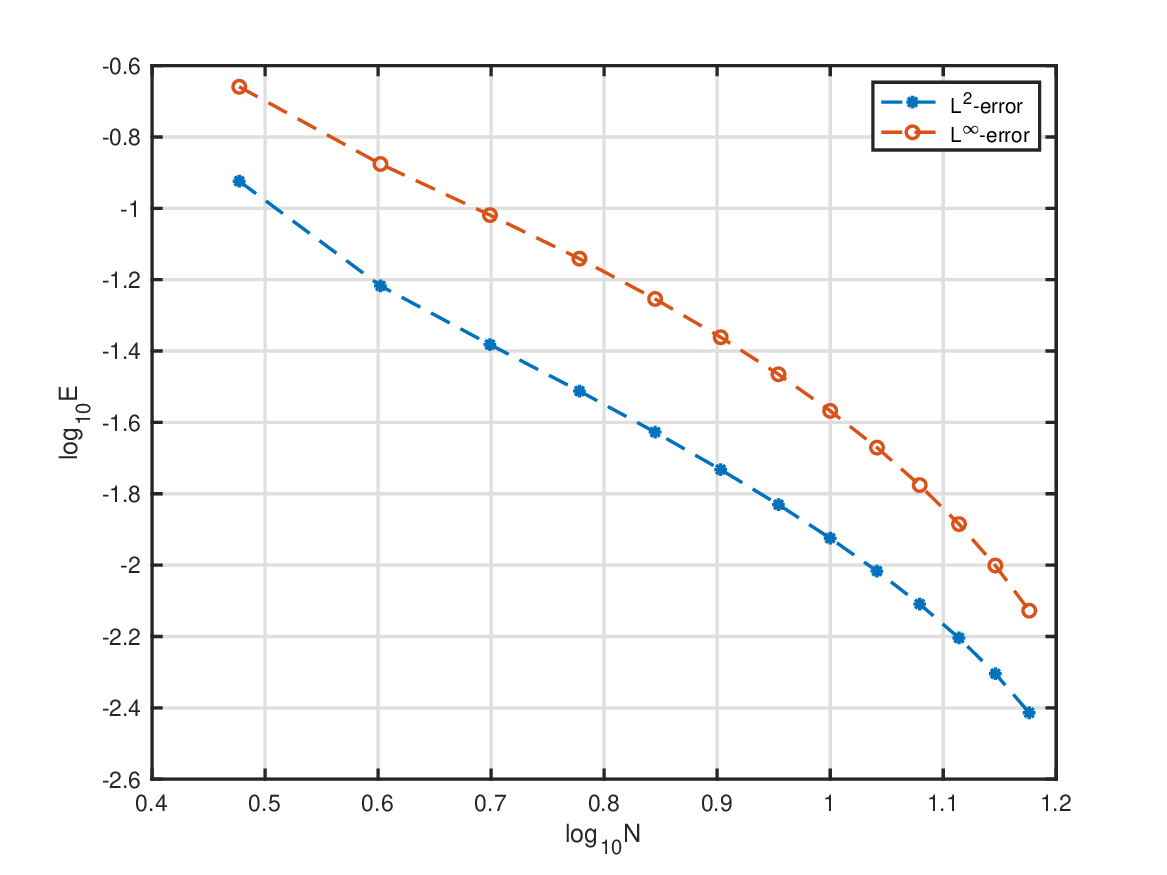}}
		\caption{\bf{Example \ref{example_second_NO3} } with $\lambda=1$}
		\label{Fig_second_NO3_1}
	\end{figure}
	
	\begin{table}[!ht]
		\centering
		\caption{\bf{Example \ref{example_second_NO3} } with $\lambda=\frac{1}{2}$: $\left\|e(\theta)\right\|_{0,\omega^{\alpha,\beta,1}}$ and $\|e(\theta)\|_{\infty}$.}
		\begin{tabular}{llllll}
			\hline$N$ & 8 & 10 & 12 & 14 & 16\\
			\hline$L^2$-error & $2.03648 \mathrm{e}-04$ & $1.15520\mathrm{e}-05$ & $5.52804 \mathrm{e}-07$ & $8.95862 \mathrm{e}-09$ & $4.80398 \mathrm{e}-09$\\
			$L^{\infty}$-error & $4.43349\mathrm{e}-04$ & $2.49200\mathrm{e}-05$ & $1.20951 \mathrm{e}-06$ &$9.79723\mathrm{e}-09$ &$7.14634 \mathrm{e}-09$\\
			\hline
		\end{tabular}
		\label{tabular_second_NO3_e}
	\end{table}
	
	\begin{table}[!ht]
		\centering
		\caption{\bf{Example \ref{example_second_NO3}} with  $\lambda=\frac{1}{2}$: $\left\|e^*(\theta)\right\|_{0,\omega^{\alpha,\beta,1}}$ and $\|e^*(\theta)\|_{\infty}$.}
		\begin{tabular}{llllll}
			\hline$N$ & 8 & 10 & 12 & 14 & 16\\
			\hline$L^2$-error & $3.20319 \mathrm{e}-04$ & $3.09893\mathrm{e}-05$ & $8.03952 \mathrm{e}-07$ & $3.91366\mathrm{e}-07$ & $9.18374 \mathrm{e}-09$\\
			$L^{\infty}$-error & $7.07185\mathrm{e}-04$ & $6.81766\mathrm{e}-05$ & $1.80248 \mathrm{e}-06$ &$8.74139\mathrm{e}-07$ &$9.61830 \mathrm{e}-09$\\
			\hline
		\end{tabular}
		\label{tabular_second_NO3_e*}
	\end{table}

\end{example}

\begin{example}\label{example_second_NO4}
	Now a problem with the unknown exact solution is exhibited\cite{WVIDEs_hp_QinYu}:
		\begin{equation}
		\left\{\begin{array}{l}
			y^{\prime}(t)=\cos(t)y(t)+e^{-t}y(\varepsilon t)+\sin (2 t)-\int_0^t(t-s)^{-\mu}(1+\sin (t s)) y(s) d s \\
			\qquad \quad-\int_0^{\varepsilon t}(\varepsilon t-\tau)^{-\mu}(1+\cos(t \tau))y(\tau) d \tau, t \in[0, \frac{1}{2}] \\
			y(0)=3.
		\end{array}\right.
	\end{equation}
where $\mu=\frac{1}{2}, \varepsilon=0.5$ and the reference "exact" solution is computed by $\lambda=\frac{1}{2}$ and $N=18$, then we choose $\lambda=\frac{1}{2}$ for the numerical solution. From Remark \ref{remark_form of solution},  we infer that $y\left(t^{\frac{1}{\lambda}}\right)$ and $y^{\prime}\left(t^{\frac{1}{\lambda}}\right)$ are so smooth that the exponential convergence results can be observed in Figure \ref{Fig_second_NO4_0.5}, whereas an algebraic convergence results are exhibited in Figure \ref{Fig_second_NO4_1}, it converges faster than the situation $\lambda=1$.
	\begin{figure}
		\subfloat[$\lambda=\frac{1}{2},\left\|e(\theta)\right\|_{0,\omega^{\alpha,\beta,1}},\|e(\theta)\|_{\infty}$]{\includegraphics[width=0.5\textwidth]{ 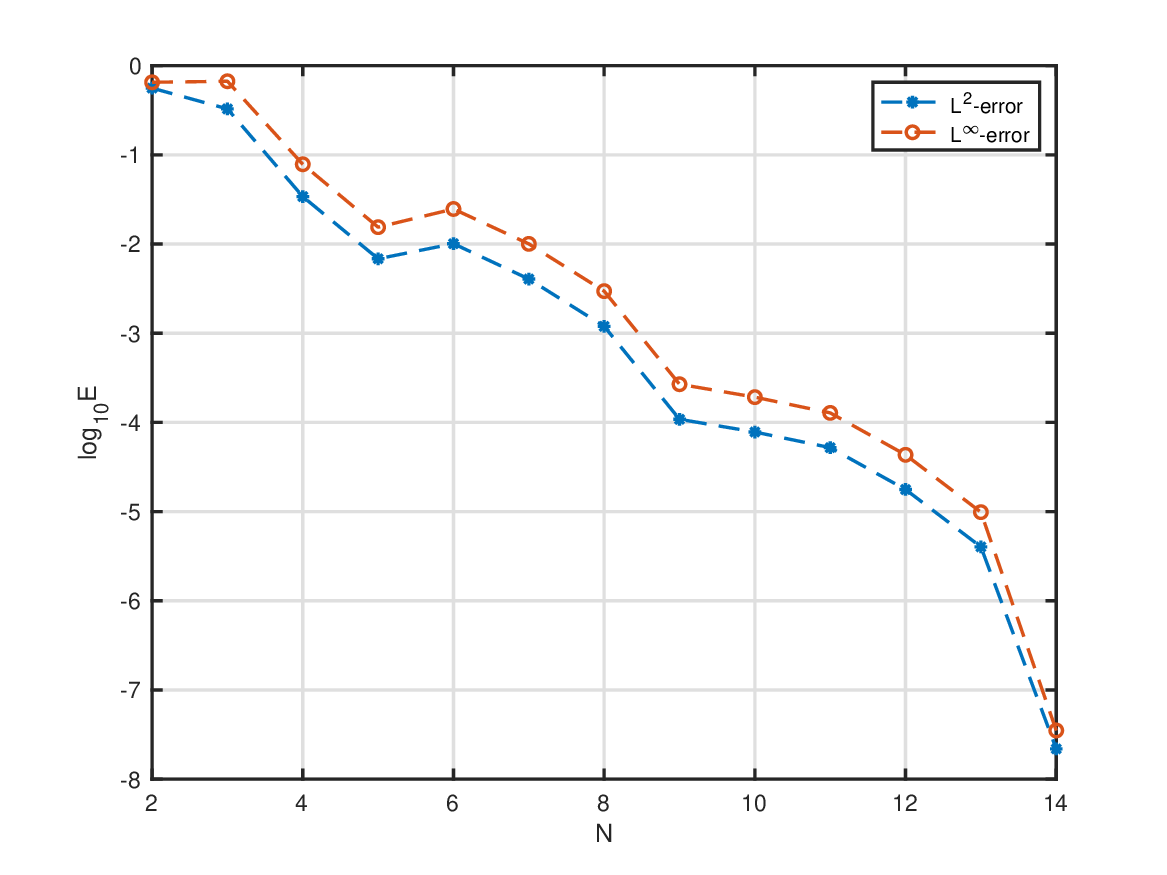}}%
		\hfill
		\subfloat[$\lambda=\frac{1}{2},\left\|e^*(\theta)\right\|_{0,\omega^{\alpha,\beta,1}},\|e^*(\theta)\|_{\infty}$]{\includegraphics[width=0.5\textwidth]{ 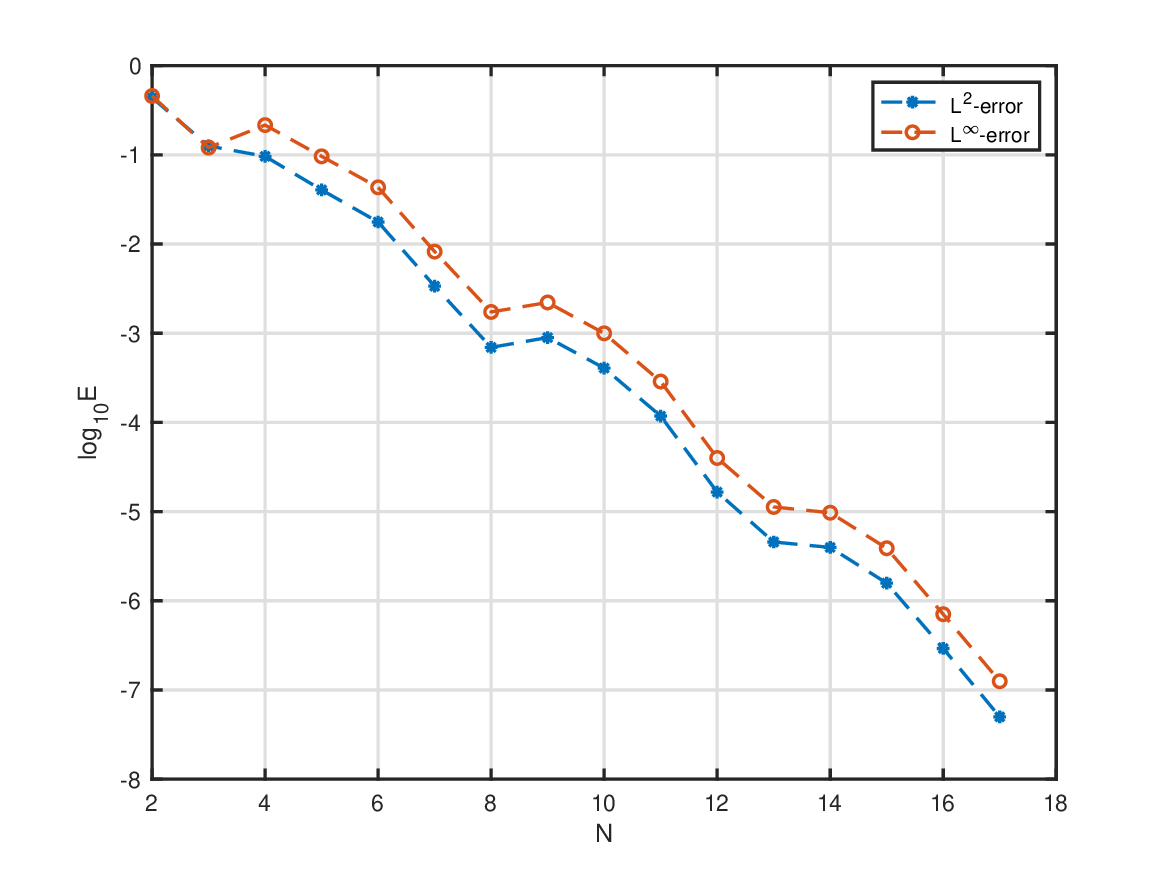}}
		\caption{\bf{Example \ref{example_second_NO4} } with $\lambda=\frac{1}{2}$}
		\label{Fig_second_NO4_0.5}
	\end{figure}
	
	\begin{figure}
		\subfloat[$\lambda=1,\left\|e(\theta)\right\|_{0,\omega^{\alpha,\beta,1}},\|e(\theta)\|_{\infty}$]{\includegraphics[width=0.5\textwidth]{ 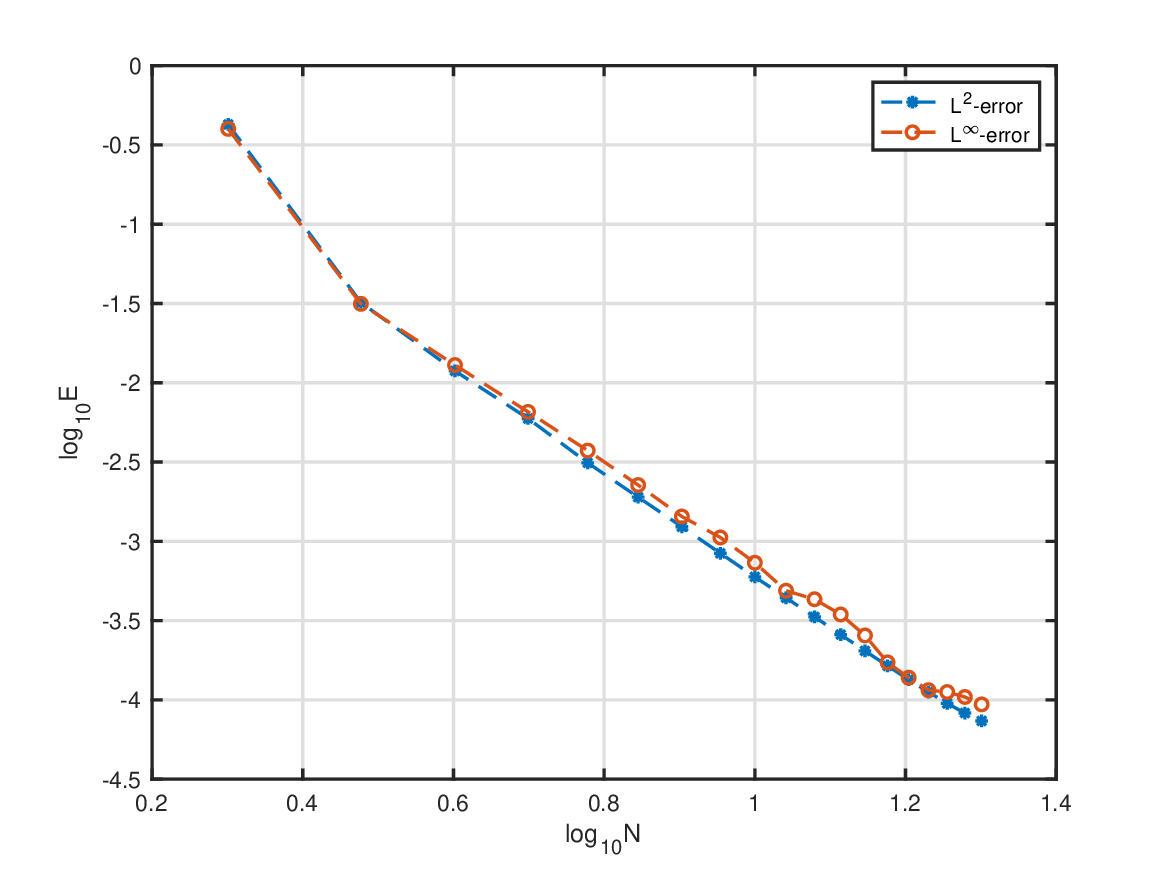}}%
		\hfill
		\subfloat[$\lambda=1,\left\|e^*(\theta)\right\|_{0,\omega^{\alpha,\beta,1}},\|e^*(\theta)\|_{\infty}$]{\includegraphics[width=0.5\textwidth]{ 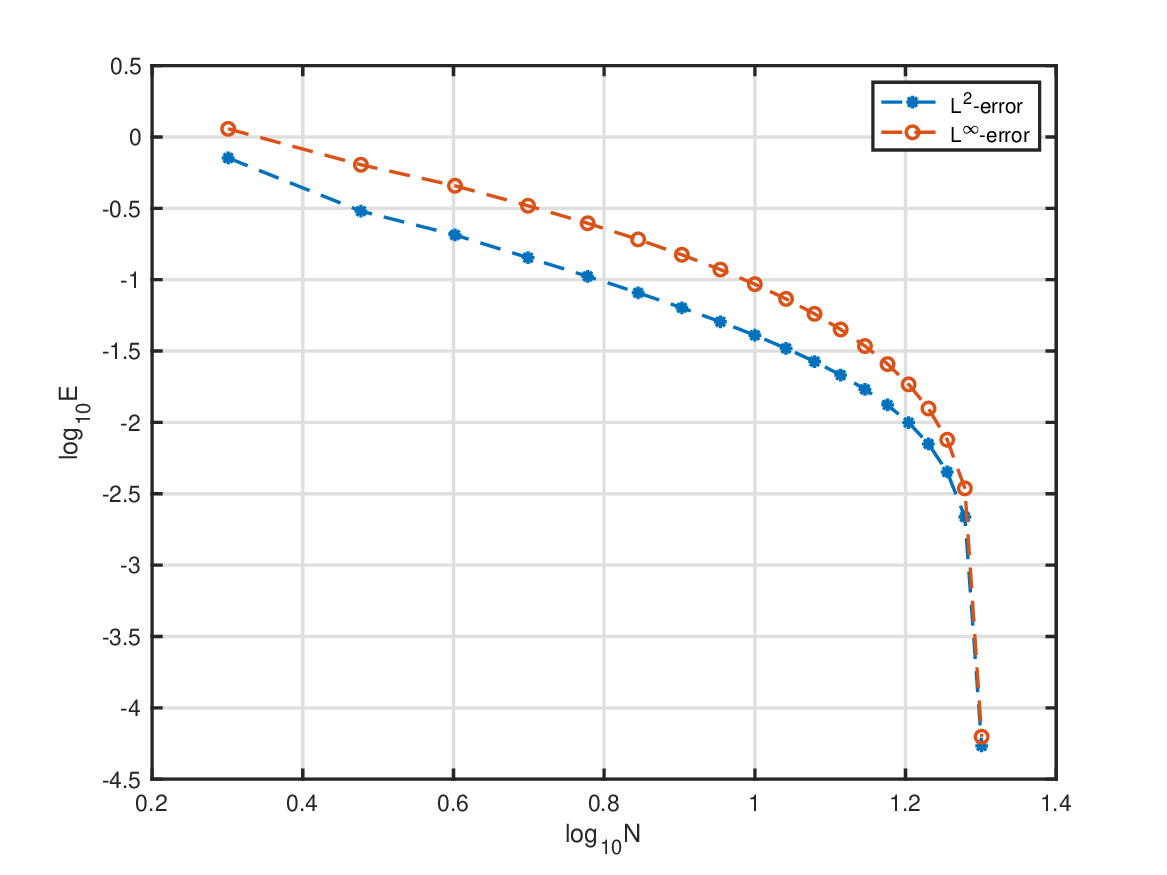}}
		\caption{\bf{Example \ref{example_second_NO4} } with $\lambda=1$}
		\label{Fig_second_NO4_1}
	\end{figure}
	
	\begin{table}[!ht]
	\centering
	\caption{\bf{Example \ref{example_second_NO4} } with $\lambda=\frac{1}{2}$: $\left\|e(\theta)\right\|_{0,\omega^{\alpha,\beta,1}}$ and $\|e(\theta)\|_{\infty}$.}
	\begin{tabular}{llllll}
		\hline$N$ & 6 & 8 & 10 & 13 & 14\\
		\hline$L^2$-error & $1.00814\mathrm{e}-02$ & $1.19409 \mathrm{e}-03$ & $7.80140 \mathrm{e}-05$ & $4.02088 \mathrm{e}-06$ & $2.18681 \mathrm{e}-08$\\
		$L^{\infty}$-error & $2.47012\mathrm{e}-02$ & $2.96236\mathrm{e}-03$ & $1.91946 \mathrm{e}-04$ &$9.85380\mathrm{e}-06$ &$3.51495 \mathrm{e}-08$\\
		\hline
	\end{tabular}
	\label{tabular_second_NO4_e}
\end{table}

\begin{table}[!ht]
	\centering
	\caption{\bf{Example \ref{example_second_NO4}} with  $\lambda=\frac{1}{2}$: $\left\|e^*(\theta)\right\|_{0,\omega^{\alpha,\beta,1}}$ and $\|e^*(\theta)\|_{\infty}$.}
	\begin{tabular}{llllll}
		\hline$N$ & 6 & 7 & 10 & 12 & 14\\
		\hline$L^2$-error & $1.76512\mathrm{e}-02$ & $3.36748\mathrm{e}-03$ & $4.06070\mathrm{e}-04$ &$1.65426 \mathrm{e}-05$ &$3.96338 \mathrm{e}-06$\\
		
		$L^{\infty}$-error & $4.31662\mathrm{e}-02$ & $8.22254 \mathrm{e}-03$ & $9.96687 \mathrm{e}-04$ & $3.98986 \mathrm{e}-05$ & $9.71171 \mathrm{e}-06$\\
		\hline
	\end{tabular}
	\label{tabular_second_NO4_e*}
\end{table}
\end{example}

\section{Conclusions}\label{section_Conclusions}
We propose and analyze a fractional Jacobi-spectral-collocation approximation for the second kind Volterra integro-differential equations with weakly singular kernels and proportional delays. Firstly, we introduced a fractional numerical method and proved the error estimates. With the suitable $\lambda$, the exponential convergence rate can be achieved after a variable change $t\rightarrow t^{\frac{1}{\lambda}}$ in order that the typical solutions $y(t)$ and its derivative $y^{\prime}(t)$  become analytical. Finally, numerical results demonstrate the theoretical proof and the efficient of the proposed method.


\begin{thebibliography}{30}
	\bibitem{application_population}
	K.G. TeBeest, Numerical and analytical solutions of Volterra's population model, \textit{SIAM Rev.},\textbf{39}:3 (1997), 484–493.
	
	\bibitem{application_粘弹性现象}
	M. Renardy, W. J. Hrusa, and J. A. Nohel, Mathematical Problems in Viscoelasticity, \textit{Pitman Monographs and Surveys in Pure and Applied Mathematics}, Vol. 35 (1987), John Wiley \& Sons.
	
	\bibitem{spectral_2}
	C. Huang, T. Tang, and Z. Zhang, Supergeometric convergence of spectral collocation methods for weakly singular Volterra and Fredholm integral equations with smooth solutions, \textit{Journal of Computational Mathematics}, \textbf{29}:6 (2011), 698–719.
	
	\bibitem{spectal_TangTao}
	Y. Chen and T. Tang, Spectral methods for weakly singular Volterra integral equations with smooth solutions, \textit{Journal of Computational and Applied Mathematics}, \textbf{233}:4 (2009), 938-950.
	
	\bibitem{spectral_Chen_WVIDEs}
	Y. Wei and Y. Chen, Convergence analysis of the spectral methods for weakly singular Volterra integro-differential equations with smooth solutions, \textit{Advances in Applied Mathematics and Mechanics}, \textbf{4}:1 (2012), 1-20.
	
	\bibitem{spectral_Chen_VIDEs_delay}
	Y. Wei and Y. Chen, Legendre spectral collocation methods for proportional Volterra delay-integro-differential equations, \textit{Journal of Scientific Computing}, \textbf{53} (2012), 672-688.
	
	\bibitem{spectral_Chen_WVIDEs_delay}
	X. Shi and Y. Chen, Spectral-Collocation Method for Volterra Delay Integro-Differential Equations with Weakly Singular Kernels, \textit{Journal of Computational and Applied Mathematics}, \textbf{392} (2020), 113458.
	
	\bibitem{spectral_4}
	X. Tao, Z. Xie, and X. Zhou, Spectral Petrov-Galerkin methods for the second kind Volterra type integro-differential equations, \textit{Numerical Mathematics: Theory, Methods and Applications}, \textbf{4}:2 (2011), 216–236.
	
	\bibitem{spectral_VT_1}
	P. Baratella, A. Palamara Orsi, Numerical solution of weakly singular linear Volterra integro-differential equations, \textit{Computing}, \textbf{77}:1 (2006), 77--96.
	
	\bibitem{spectral_VT_2}
	T. Diogo, P.M. Lima, A. Pedas, G. Vainikko, Smoothing transformation and spline collocation for weakly singular Volterra integro-differential equations, \textit{Appl. Numer. Math.}, \textbf{114} (2017), 63--76.
	
	\bibitem{spectral_VT_3}
	X. Shi, Y. Wei, F. Huang, Spectral collocation methods for nonlinear weakly singular Volterra integro-differential equations, \textit{Numer. Methods Partial Differential Equations}, \textbf{35}:2 (2019), 576--596.
	
	\bibitem{delay_Sheng_1}
	Z. Wang, C. Sheng, H. Jia, et al., A Chebyshev spectral collocation method for nonlinear Volterra integral equations with vanishing delays, \textit{East Asian J. Appl. Math.}, \textbf{8}:2 (2018), 233--260.
	
	\bibitem{delay_Sheng_2}
	C. Sheng, Z. Wang, B. Guo, An hp-spectral collocation method for nonlinear Volterra functional integro-differential equations with delays, \textit{Applied Numerical Mathematics}, \textbf{105} (2016), 1--24.
	
	\bibitem{delay_Sheng_3}
	Z. Wu, C. Shih, An hp-spectral collocation method for nonlinear Volterra integral equations with vanishing variable delays, \textit{Mathematics of Computation}, \textbf{85}:298 (2016), 635--666.
	
	\bibitem{VIDEs_delays_1}
	J. Zhao, Y. Cao, and Y. Xu, Sinc numerical solution for proportional Volterra delay-integro-differential equation, \textit{International Journal of Computer Mathematics}, \textbf{94}:5 (2017), 853-865.
	
	\bibitem{VIDEs_delays_2}
	J. Zhao, Y. Cao, and Y. Xu, Tau approximate solution of linear proportional Volterra delay-integro-differential equation, \textit{Computational and Applied Mathematics}, \textbf{39} (2020), 1-15.
	
	\bibitem{VIDEs_delays_3}
	L. Yi and B. Guo, The h-p version of the continuous Petrov–Galerkin method for nonlinear Volterra functional integro-differential equations with vanishing delays, \textit{International Journal of Numerical Analysis and Modeling}, \textbf{15}:1 
	
	\bibitem{VIDEs_delays_4}
	L. Wang and L. Yi, An h-p version of the continuous Petrov-Galerkin method for Volterra delay-integro-differential equations, \textit{Advances in Computational Mathematics}, \textbf{43} (2017), 1437-1467.
	
	\bibitem{VIDEs_delays_5}
	L. Wang and L. Yi, An h-p version of the discontinuous Galerkin method for Volterra integro-differential equations with vanishing delays, \textit{Journal of Scientific Computing}, \textbf{81}:3 (2019), 2303-2330.
	
	\bibitem{WVIDEs_hp_MaZheng}
	Z. Ma and C. Huang, An hp-version fractional collocation method for Volterra integro-differential equations with weakly singular kernels, \textit{Numerical Algorithms}, \textbf{92}:4 (2023), 2377-2404.
	
	\bibitem{WVIDEs_hp_QinYu}
	Y. Qin and C. Huang, An hp-version error estimate of spectral collocation methods for weakly singular Volterra integro-differential equations with vanishing delays, \textit{Computational and Applied Mathematics}, \textbf{43}:5 (2024), 301.
	
	\bibitem{spectral_3}
	D. Hou and C. Xu, A fractional spectral method with applications to some singular problems, \textit{Advances in Computational Mathematics}, \textbf{43}:5 (2017), 911–944.
	
	\bibitem{Muntz}
	D. Hou, Y. Lin, M. Azaiez, et al., A Müntz-collocation spectral method for weakly singular Volterra integral equations, \textit{Journal of Scientific Computing}, \textbf{81}:3 (2019), 2162-2187.
	
	\bibitem{3VIEs_MaZheng}
	Z. Ma and C. Huang, Fractional collocation method for third-kind Volterra integral equations with nonsmooth solutions, \textit{Journal of Scientific Computing}, \textbf{95}:1 (2023), 26.
	
	\bibitem{CVIDEs_MaZheng}
	Z. Ma, C. Huang, and A. A. Alikhanov, Error analysis of fractional collocation methods for Volterra integro-differential equations with noncompact operators, \textit{Journal of Scientific Computing}, \textbf{95}:1 (2023), 26.
	
	\bibitem{Lemma_linear_operator_1}
	D. Ragozin, Polynomial approximation on compact manifolds and homogeneous spaces, \textit{Transactions of the American Mathematical Society}, \textbf{150}:1 (1970), 41–53.
	
	\bibitem{Lemma_linear_operator_2}
	D. Ragozin, Constructive polynomial approximation on spheres and projective spaces, \textit{Transactions of the American Mathematical Society}, \textbf{162} (1971), 157–170.
	
	\bibitem{Ma毕业论文}
	Z. Ma, High Order Numerical Methods for Weakly Singular Volterra Integral and Integro-differential Equations, \textit{Ph.D. dissertation}, Huazhong University of Science and Technology, 2023. DOI: 10.27157/d.cnki.ghzku.2023.000001.
	
	\bibitem{Brunner}
	H. Brunner, Collocation methods for Volterra integral and related functional differential equations, \textit{Cambridge University Press}, 2004.
	
	\bibitem{Remark_2}
	A. Azizipour and S. Shahmorad, A new tau-collocation method with fractional basis for solving weakly singular delay Volterra integro-differential equations, \textit{Journal of Applied Mathematics and Computing}, \textbf{68}:4 (2022), 2435--2469. 
\end{thebibliography}
\end{document}